\documentclass[12pt]{amsart}
\usepackage{amsfonts,amsthm,amsmath,amssymb}
\usepackage{bbm,graphicx}
\usepackage[dvipsnames]{xcolor}
\usepackage[normalem]{ulem}
\usepackage{mathtools}
\usepackage{hyperref}

\usepackage{paralist}
\usepackage[shortlabels]{enumitem}
\usepackage{bm}
\usepackage{textcomp}
\usepackage{vmargin}
\setpapersize{USletter}
\setmargrb{2.5cm}{2cm}{2.5cm}{2.cm} % --- sets all four margins.  Cool.
\newtheorem{theorem}{Theorem}
\newtheorem{proposition}[theorem]{Proposition}

\newtheorem{lemma}[theorem]{Lemma}

\theoremstyle{definition}
\newtheorem{example}[theorem]{Example} 
\newtheorem{remark}[theorem]{Remark} 
 
\newtheorem{algorithm}[theorem]{Algorithm}
\newcommand{\CC}{{\mathbb C}}
\newcommand{\NN}{{\mathbb N}}
\newcommand{\PP}{{\mathbb P}}

\newcommand{\RR}{{\mathbb R}}
\newcommand{\TT}{{\mathbb T}}
\newcommand{\ZZ}{{\mathbb Z}}
\newcommand{\bbI}{\mathbbm{1}}

\newcommand{\ev}{\mathrm{ev}}

\newcommand{\calA}{{\mathcal A}}
\newcommand{\calB}{{\mathcal B}}
\newcommand{\calC}{{\mathcal C}}
\newcommand{\calF}{{\mathcal F}}
\newcommand{\calG}{{\mathcal G}}
\newcommand{\calL}{{\mathcal L}}
\newcommand{\QS}{\mathcal{QS}}

\newcommand{\calV}{{\mathcal V}}
\newcommand{\calX}{{\mathcal X}}
\newcommand{\calY}{{\mathcal Y}}

\newcommand{\sect}{Section}
\newcommand{\theo}{Theorem}

\newcommand{\chap}{Chapter}
\newcommand{\algo}{Algorithm}
\newcommand{\lemm}{Lemma}

\newcommand{\pr}{{p}{r}}
\DeclareMathOperator{\Proj}{{\rm Proj}}
\DeclareMathOperator{\Hom}{{\rm Hom}}

\DeclareMathOperator{\ini}{{\rm in}}
\DeclareMathOperator{\lt}{{\rm lt}}
\DeclareMathOperator{\gr}{{\rm gr}}
\DeclareMathOperator{\sat}{{\rm sat}}
\DeclareMathOperator{\NO}{{\rm NO}}

\newcommand{\defcolor}[1]{{\color{RoyalBlue}#1}}
\newcommand{\demph}[1]{\defcolor{{\sl #1}}}
\title{Numerical homotopies from Khovanskii bases}
\author{M.~Burr}
\address{Michael Burr, School of Mathematical and Statistical Sciences, 
            Clemson University, 220 Parkway Drive, Clemson, SC 29634-0975,  USA}
\email{burr2@clemson.edu}
\urladdr{https://cecas.clemson.edu/\~{}burr2/}
\author{F.~Sottile}
\address{Frank Sottile, Department of Mathematics,
         Texas A\&M University, College Station, Texas 77843,  USA}
\email{sottile@math.tamu.edu}
\urladdr{http://www.math.tamu.edu/\~{}sottile}
\author{E.~Walker}
\address{Elise Walker, Department of Mathematics,
         Texas A\&M University, College Station, Texas 77843,  USA}
\email{walkere@math.tamu.edu}
\urladdr{http://www.math.tamu.edu/\~{}walkere}
\thanks{Research of Sottile and Walker supported in part by NSF grant DMS-1501370 and
        Simons Collaboration Grant for Mathematics number 636314}
\thanks{Research of Burr supported in part by NSF grants CCF-1527193 and DDF-1913119}
\thanks{This paper began while the authors were visiting the ICERM for the semester program on Nonlinear Algebra in Fall 2018}
\begin{document}

%%%%%%%%%%%%%%%%%%%%%%%%%%%%%%%%%%%%%%%%%%%%%%%%%%%%%%%%%%%%%%%%%%%%%%%%%%%%%%%%%
\begin{abstract}
  We present numerical homotopy continuation algorithms for solving systems of equations on a variety
  in the presence of a finite Khovanskii basis.  These take 
  advantage of Anderson's flat degeneration to a toric variety.  When Anderson's
  degeneration embeds into projective space, our  algorithm is a special case of a
  general toric two-step homotopy algorithm.  
  When Anderson's degeneration is embedded in a weighted projective space, we explain how to lift to a projective space and
  construct an appropriate modification of the toric homotopy. 
  Our algorithms are illustrated on several examples using \texttt{Macaulay2}.
\end{abstract}
%%%%%%%%%%%%%%%%%%%%%%%%%%%%%%%%%%%%%%%%%%%%%%%%%%%%%%%%%%%%%%%%%%%%%%%%%%%%%%%%%
\maketitle

%%%%%%%%%%%%%%%%%%%%%%%%%%%%%%%%%%%%%%%%%%%%%%%%%%%%%%%%%%%%%%%%%%%%%%%%%%%%%%%%%

We consider the problem of computing the isolated solutions to the system
 \begin{equation}\label{Eq:generalSystem}
   f_1(z)=f_2(z)=\dotsb=f_d(z)=0,
 \end{equation}
where $f_1,\dots,f_d$ are general members of a finite-dimensional vector space $V$ of rational 
functions on a complex algebraic variety $X$ of dimension $d$.
Kaveh-Khovanskii~\cite{KaKha,KaKhb} and Lazars\-feld-Musta\c{t}\v{a}~\cite{LaMu}
show that the number of solutions is the normalized volume of the Newton-Okounkov body associated to $V$.
The accompanying theory extends many uses of Newton polytopes from toric varieties to general algebraic varieties.  
This theory lends itself to algorithms when $V$ has a finite Khovanskii basis~\cite{KM_tropical}.

The evaluation of functions in $V$ induces the rational Kodaira map $\varphi\colon X\dasharrow\PP(V^*)$.
The solutions to System~\eqref{Eq:generalSystem} are the pull backs  
of the points of a linear section $\varphi(X)\cap L$ along $\varphi$.
When $V$ has a finite Khovanskii basis, Anderson~\cite{Anderson} shows that (the closure of)
$\varphi(X)$ has a flat degeneration to a toric variety associated to the Newton-Okounkov body of $V$.
We describe numerical algorithms for computing a linear section based on this toric
degeneration and the polyhedral homotopy algorithm~\cite{HS95,VVC}.
Solving System~\eqref{Eq:generalSystem} then requires computing the pull back of the linear section.
 
Our numerical algorithms for computing a linear section are based on homotopy continuation~\cite{Morgan}.
This approach uses path tracking from numerical analysis to compute the solutions to a target system $F$ given all
solutions to a start system $G$ along with a homotopy interpolating the two systems.
Anderson's flat toric degeneration gives a homotopy where the start system is a linear section of a toric variety and the
target system is a linear section of $\varphi(X)$.
Flatness guarantees that the number of solutions to the start and target systems are equal.
Thus the homotopy is optimal in the sense that no extraneous paths are tracked.

Our start system is a linear section of a toric variety, which may be solved using the optimal
polyhedral homotopy algorithm~\cite{HS95,VVC}.
Beyond those derived from a finite Khovanskii basis, there are many instances in which a projective variety has a flat
degeneration into a toric variety. 
In Section~\ref{S:TDHC}, we describe an optimal toric two-step homotopy algorithm for solving systems given a toric
degeneration in an ambient projective space. We also present examples of such flat degenerations into toric varieties.

When the Khovanskii basis is a subset of $V$, 
Anderson's degeneration may be embedded in the projective space $\PP(V^*)$.
In Section~\ref{S:Khovanskii}, we present the Khovanskii homotopy algorithm, which uses this embedding and the toric
two-step homotopy to solve System~\eqref{Eq:generalSystem}.
For a general Khovanskii basis, Anderson's degeneration may only be embedded in a weighted projective space and
System~\eqref{Eq:generalSystem} is not a pull back of a general linear section.
In Section~\ref{S:WPS}, we describe how to adapt the toric algorithm to this general case of a Khovanskii basis.

We end each section with a concrete example to illustrate our techniques and algorithms.
These examples are computed with \texttt{Macaulay2} scripts~\cite{M2}, which are archived on \texttt{GitHub}:
 \begin{center}
    \url{https://github.com/EliseAWalker/KhovanskiiHomotopy/}
 \end{center}
We use the \texttt{NumericalAlgebraicGeometry} package~\cite{NAG4M2} to call the 
packages
\texttt{Bertini}~\cite{Bertini} and  \texttt{PHCpack}~\cite{PHCpack} for user-defined homotopies and the polyhedral
homotopy, respectively.  We discuss practical issues that arise from using these software packages in Section~\ref{S:Exp}.

%%%%%%%%%%%%%%%%%%%%%%%%%%%%%%%%%%%%%%%%%%%%%%%%%%%%%%%%%%%%%%%%%%%%%%%%%%%%%%%%%
\section{Homotopy continuation and toric degenerations}\label{S:TDHC}

Numerical homotopy continuation computes the solutions to a system $F$ of polynomial equations given the solutions to a 
related
system $G$.
This method uses numerical path tracking along a homotopy, which is a family of systems containing both $F$ and $G$.
We begin by reviewing homotopy continuation and then discuss how flat families are a source of homotopies.
When a flat family is a degeneration into a toric variety and the system $G$ is a linear section of that toric
variety, we describe the toric two-step homotopy algorithm whose first step is a polyhedral homotopy.

%%%%%%%%%%%%%%%%%%%%%%%%%%%%%%%%%%%%%%%%%%%%%%%%%%%%%%%%%%%%%%%%%%%%%%%%%%%%%%%%%
\subsection{Numerical homotopy continuation}

Numerical homotopy continuation is a method for solving a system $F(x)=0$ of polynomial equations~\cite{Morgan}.
It uses a one-parameter family $\defcolor{H(x;t)}$ (in $t\in\CC_t$) of polynomial systems called a \demph{homotopy}. 
Numerical homotopy continuation mandates that the \demph{start system} $\defcolor{G(x)}\vcentcolon=  H(x;0)$ has known
solutions and the solutions to the \demph{target system} $F(x)$ are among those to $H(x;1)$.
We further require that $H(x;t)$ defines a curve \defcolor{$C$} in $\CC^n_x\times\CC_t$ (or $\PP^n_x\times\CC_t$)
with $t=0$ a regular value of the projection $C\to \CC_t$.
These assumptions imply that there are enough solutions to the start system so that along
a general path in $\CC_t$, the solutions to the start system $H(x;0)$ deform to the solutions to the system $H(x;1)$.
The homotopy is \demph{optimal} if every solution to the
start system deforms to a distinct solution to the target system
so that no extraneous paths are tracked. 

Given a homotopy $H$, let $\gamma$ be a general arc in $\CC_t$ between $0$ and $1$.
The restriction of $H$ to $\gamma$ is a family of smooth arcs. 
Standard numerical path tracking algorithms starting with solutions to $G$ can compute the set of all solutions to $H(x;1)$, 
which includes all solutions to $F$.
When $t=1$ is also a regular value of the projection $C\to \CC_t$ and $H(x;1)=F(x)$, the homotopy is optimal.

Several software packages implement numerical homotopy continuation methods.
These include \texttt{Bertini}~\cite{Bertini}, \texttt{NumericalAlgebraicGeometry}~\cite{NAG4M2},
\texttt{HomotopyContinuation.jl}~\cite{BreidingTimme}, 
\texttt{HOM4PS}~\cite{HOM4PS}, and \texttt{PHCpack}~\cite{PHCpack}.
The first three implement user-defined homotopies, and our computational examples use the user-defined homotopy method
provided in \texttt{Bertini}.
The last three packages implement the polyhedral homotopy method~\cite{HS95,VVC}, and we use \texttt{PHCpack} 
for solving systems of sparse polynomials coming from linear sections of toric varieties.

%%%%%%%%%%%%%%%%%%%%%%%%%%%%%%%%%%%%%%%%%%%%%%%%%%%%%%%%%%%%%%%%%%%%%%%%%%%%%%%%%

\subsection{Homotopies from flat families}

Suppose that $X\subset\PP^n$ is a subvariety of dimension $d$.
A \demph{linear section} of $X$ is a transverse intersection $X\cap L$ where $L\subset\PP^n$ is a linear subspace of
codimension $d$ so that $X\cap L$ consists of $\deg X$ points. 

Let $\defcolor{\calX}\subset\PP^n\times\CC$ be a variety with a surjective map $\pi\colon\calX\to\CC$.
Then $\pi$ realizes $\calX$ as a family of projective varieties over $\CC$ where a point $t\in\CC$ corresponds to the
fiber $\defcolor{\calX_t}\vcentcolon= \pi^{-1}(t)\subset\PP^n$. 
There is an open subset $U\subset\CC$ such that $\calX$ is flat over $U$.
Flatness is an algebraic property
which captures the geometric notion that the fibers $\calX_t$ vary continuously with $t\in U$ \cite[\chap\ 6]{Eisenbud}.
For example, the fibers of a flat family all have the same dimension and degree.

Suppose that the fibers of a flat family $\calX$ over $U\subset\CC$ have dimension $d$ and that $0,1\in U$.
Let $L\subset\PP^n$ be a general linear subspace of codimension $d$ which meets both $\calX_0$ and $\calX_1$ transversally
so that  $\calX_0\cap L$ and $\calX_1\cap L$ are linear sections.
Let $H(x;t)$ be finitely many polynomials defining $\calX$ and $d$ linear forms defining $L$. 
We call $H(x;t)$ a \demph{linear section homotopy}.  

%%%%%%%%%%%%%%%%%%%%%%%%%%%%%%%%%%%%%%%%%%%%%%%%%%%%%%%%%%%%%%%%%%%%%%%%%%%%%%%%%
\begin{proposition}[Linear section homotopy]
  \label{prop:FlatHomotopy}
  A linear section homotopy $H(x;t)$ is an optimal homotopy with start
  system  $\calX_0\cap L$ and target system $\calX_1\cap L$.  
\end{proposition}
%%%%%%%%%%%%%%%%%%%%%%%%%%%%%%%%%%%%%%%%%%%%%%%%%%%%%%%%%%%%%%%%%%%%%%%%%%%%%%%%%

%%%%%%%%%%%%%%%%%%%%%%%%%%%%%%%%%%%%%%%%%%%%%%%%%%%%%%%%%%%%%%%%%%%%%%%%%%%%%%%%%
\begin{proof}
 Let \defcolor{$C$} be the union of components of $\calX\cap L$ that contain both $\calX_0\cap L$
 and $\calX_1\cap L$.
 Since these intersections are zero-dimensional, $C$ is a
 curve.  
 Furthermore, both $t=0$ and $t=1$ are regular values of the projection $C\to\CC_t$.
 Thus, $H$ is a homotopy.
 Flatness implies that $\calX_0\cap L$ and $\calX_1\cap L$ have the same number of points so that the homotopy $H$ is
 optimal. 
\end{proof}
%%%%%%%%%%%%%%%%%%%%%%%%%%%%%%%%%%%%%%%%%%%%%%%%%%%%%%%%%%%%%%%%%%%%%%%%%%%%%%%%%

A linear section is part of a \demph{witness set},
which is a fundamental data structure in numerical algebraic geometry~\cite{Bertini_Book}.
Specifically, a witness set for a $d$-dimensional variety $X\subset\PP^n$ is a triple $(F,L,X\cap L)$ where $F$ is a set of
homogeneous polynomials (forms) on $\PP^n$ defining $X$,
$L$ is a set of $d$ general linear forms 
defining a linear subspace (which is also written $L$), and $X\cap L$ is the corresponding linear
section. 

In the linear section homotopy in Proposition \ref{prop:FlatHomotopy}, 
$L$ is a fixed general linear space
and the variety $\calX_t$ moves.
Our algorithms sometimes require linear spaces which are not general.
For this, we use a homotopy where the variety is fixed, but the linear section moves, 
which is described in the following basic algorithm for moving a witness set:

%%%%%%%%%%%%%%%%%%%%%%%%%%%%%%%%%%%%%%%%%%%%%%%%%%%%%%%%%%%%%%%%%%%%%%%%%%%%%%%%%
\begin{algorithm}[Witness Set Homotopy]\ 
   \label{Alg:WSH}

   {\bf Input:}
   A witness set $(G,L,X\cap L)$ for $X$ and a codimension $d$ linear subspace $L'$ such 
   
   \mbox{{\color{white}{\bf Input:}}} that $X\cap L'$ is finite.

   {\bf Output:} The points of $X\cap L'$.
   
\pagebreak
   {\bf Do:} 
   \begin{enumerate}[(\roman*)] 
   \item Let $\defcolor{H}\vcentcolon= (G, tL' + (1-t)L)$, a homotopy with start system $X\cap L$ and
     target system $X\cap L'$.
   \item Use path tracking starting from the points of $X\cap L$ to compute the points of  $X\cap L'$. 
   \end{enumerate}
   
\end{algorithm}
%%%%%%%%%%%%%%%%%%%%%%%%%%%%%%%%%%%%%%%%%%%%%%%%%%%%%%%%%%%%%%%%%%%%%%%%%%%%%%%%%

%%%%%%%%%%%%%%%%%%%%%%%%%%%%%%%%%%%%%%%%%%%%%%%%%%%%%%%%%%%%%%%%%%%%%%%%%%%%%%%%%
%
\subsection{Toric degenerations}%\label{S:TD}

A \demph{toric degeneration} $\calX\subset\PP^n\times\CC_t$ is a flat family over $\CC_t$ whose special fiber $\calX_0$ is
a toric variety (see \cite{CLS} for details on toric varieties).  
Given a toric degeneration, the linear section homotopy leads to the toric two-step homotopy which we describe in
Algorithm~\ref{Alg:TDA}.

A vector $\defcolor{\alpha}=(a_1,\dotsc,a_d)\in\ZZ^d$ is the exponent of a Laurent monomial
$\defcolor{z^\alpha}\vcentcolon= z_1^{a_1}\dotsb z_d^{a_d}$, which is a function on the algebraic torus
$(\CC^\times)^d$, where  $\defcolor{\CC^\times}\vcentcolon=\CC\smallsetminus\{0\}$. 
Suppose that $\calA$ is a $d\times(n{+}1)$ integer matrix whose columns
$\{\alpha_0,\dotsc,\alpha_n\}$ are a finite set of $n{+}1$ integer vectors.
The  \demph{toric Kodaira map} $\varphi_\calA\colon(\CC^\times)^d\rightarrow\PP^n$ is defined by
$\defcolor{\varphi_\calA(z)}\vcentcolon=[z^{\alpha_0},\dotsc,z^{\alpha_n}]$, and the \demph{toric variety}
\defcolor{$X_\calA$} is the closure of its image. 
The homogeneous ideal of $X_\calA$ is spanned by the following 
set of binomials~\cite[\chap\ 4]{GBCP}:
\[
  \bigl\{ x^u - x^v :  \sum \alpha_i u_i = \sum \alpha_i v_i \mbox{\ and\ } \sum u_i = \sum v_i \bigr\}.
\]

We describe a variant of this Kodaira map for translated toric varieties.
The torus $(\CC^\times)^{n+1}$ acts on $\PP^n$ by independently scaling each coordinate.
This action factors through the quotient of $(\CC^\times)^{n+1}$ by its diagonal, \defcolor{$\Delta\CC^\times$}.
This quotient is the dense torus $\defcolor{\TT}$ of $\PP^n$.
For a point $p\in\TT$, let \defcolor{$p.X_\calA$} be the translation of the toric variety $X_\calA$ by $p$. 
We note that $p\in p.X_\calA$ and that $p.X_\calA=p'.X_\calA$ for any $p'\in p.X_\calA\cap\TT$.
The ideal of $p.X_\calA$ is spanned by the following set of binomials, which depend on $p$:
 \begin{equation*}%\label{Eq:toricIdeal}
  \bigl\{ p^vx^u - p^ux^v :  \sum \alpha_i u_i = \sum \alpha_i v_i \mbox{\ and\ } \sum u_i = \sum v_i \bigr\}.
 \end{equation*}
Since $p.X_\calA\simeq X_\calA$, we also call $p.X_\calA$ a toric variety and the ideal of $p.X_{\calA}$ a \demph{toric ideal}.
Writing $p=[p_0,\dotsc,p_n]$, the corresponding toric Kodaira map for $p.X_{\calA}$ is
 \begin{equation}\label{Eq:translatedToricKodaira}
  \defcolor{\varphi_{p,\calA}(z)} = [p_0z^{\alpha_0},\dotsc,p_nz^{\alpha_n}].
 \end{equation}

A linear section $p.X_\calA\cap L$ of the toric variety $p.X_\calA$ pulls back along 
$\varphi_{p,\calA}$ to the following system of sparse polynomials on $(\CC^\times)^d$ whose monomials have
exponents in $\calA$:
 \begin{equation}\label{Eq:sparse}
   g_1(z)=g_2(z) =\dots=g_d(z)=0.
 \end{equation}
 The polyhedral homotopy algorithm is an optimal homotopy for solving this system of polynomials~\cite{HS95,VVC}. 

 Let $\calX\to\CC_t$ be a toric degeneration with $d$-dimensional toric special fiber $\defcolor{p.X_\calA}=\calX_0$.
 A general linear subspace $L$ of codimension $d$ gives linear sections $p.X_\calA\cap L$ and $\calX_1\cap L$.
 We combine the linear section homotopy of Proposition~\ref{prop:FlatHomotopy} with the polyhedral homotopy to obtain
 the toric two-step homotopy algorithm for computing the points of the linear section $\calX_1\cap L$.
 Let \defcolor{$G_\calA$} be System~\eqref{Eq:sparse}, which is given by the pull back of $L$
 along  $\varphi_{p,\calA}$.
 
%%%%%%%%%%%%%%%%%%%%%%%%%%%%%%%%%%%%%%%%%%%%%%%%%%%%%%%%%%%%%%%%%%%%%%%%%%%%%%%%%
\begin{algorithm}[Toric two-step homotopy algorithm]\
\label{Alg:TDA}

{\bf Input:}  A toric degeneration $\calX\subset\PP^n\times\CC_t$ with $\calX_0=p.X_\calA$ a toric
     variety and a general

  \mbox{{\color{white}{\bf Input:}}} linear space $L\subset\PP^n$ of codimension equal to the dimension of $\calX_1$.

{\bf Output:} All points of the linear section $\calX_1\cap L$.

{\bf Do:}
\begin{enumerate}[(\roman*)]
\item Compute the system $G_{\calA}$ by pulling $L$  back along the Kodaira map $\varphi_{p,\calA}$.\label{TDA0}

\item Use the polyhedral homotopy to solve $G_\calA$.\label{TDA1}

\item Use $\varphi_{p,\calA}$ 
  to obtain the points of the linear section $p.X_\calA\cap L$.\label{TDA2}

\item Use the linear section homotopy (Proposition \ref{prop:FlatHomotopy}) beginning with the points of $p.X_\calA\cap L$ to obtain the points of
  the linear section $\calX_1\cap L$.\label{TDA3}
\end{enumerate}
\end{algorithm}
%%%%%%%%%%%%%%%%%%%%%%%%%%%%%%%%%%%%%%%%%%%%%%%%%%%%%%%%%%%%%%%%%%%%%%%%%%%%%%%%%

The discussion preceding Algorithm~\ref{Alg:TDA} justifies the following theorem:

%%%%%%%%%%%%%%%%%%%%%%%%%%%%%%%%%%%%%%%%%%%%%%%%%%%%%%%%%%%%%%%%%%%%%%%%%%%%%%%%%
\begin{theorem}\label{T:TDHA}
  Algorithm~\ref{Alg:TDA} is an optimal homotopy algorithm for computing $\calX_1\cap L$.
\end{theorem}
%%%%%%%%%%%%%%%%%%%%%%%%%%%%%%%%%%%%%%%%%%%%%%%%%%%%%%%%%%%%%%%%%%%%%%%%%%%%%%%%%

\begin{remark}\label{Ref:ToricComplex}
Algorithm~\ref{Alg:TDA} can be applied to compute $\calX_1\cap L$ 
  when the definition of a toric degeneration is relaxed so that
  $\calX_0$ is a union of toric varieties (see Remark~\ref{Ref:toriccomplex:example} for examples).
  The points in a general linear section $\calX_0\cap L$ in Algorithm~\ref{Alg:TDA} may be computed from
  systems of sparse polynomials for each toric component of $\calX_0$.\hfill$\diamond$ 
\end{remark}

%%%%%%%%%%%%%%%%%%%%%%%%%%%%%%%%%%%%%%%%%%%%%%%%%%%%%%%%%%%%%%%%%%%%%%%%%%%%%%%%%
\subsection{Examples of toric degenerations}%\label{SS:torus}

We present three examples of toric degenerations.
Example~\ref{ex:weight} is the motivating example for this paper.
Example~\ref{Ex:quasi-symmstry} illustrates an alternate source of toric degenerations.
Example~\ref{Ex:BS_TDA} is an explicit application of Algorithm~\ref{Alg:TDA}.

%%%%%%%%%%%%%%%%%%%%%%%%%%%%%%%%%%%%%%%%%%%%%%%%%%%%%%%%%%%%%%%%%%%%%%%%%%%%%%%%%
\begin{example}\label{ex:weight}
Weight degenerations induced by a $\CC^\times$-action on $\PP^n$ are a source of toric degenerations.
Anderson~\cite{Anderson} constructs a toric weight degeneration given a Khovanskii basis.
The SAGBI homotopy~\cite{HSS} is also based on a toric
weight degeneration.

We review the construction in~\cite[\sect\ 15.8]{Eisenbud} of flat families from $\CC^\times$-actions.
Let $w\in\ZZ^{n+1}$ be a weight and define an action of the torus $\CC^\times$ on $\PP^n$ by
 \[
   (x,t)\in\PP^n\times\CC^\times\mapsto\defcolor{t.x}\vcentcolon= [x_0t^{-w_0},\dotsc, x_nt^{-w_n}] \in \PP^n.
 \]
 The dual action on functions is $\defcolor{t.f(x)}\vcentcolon= f(t^{-1}.x)$, and it induces an action on polynomials.
 For a polynomial $f=\sum c_\alpha x^\alpha$,
 \begin{equation}\label{Eq:w-action}
   t.\left(\sum c_\alpha x^\alpha\right) = \sum c_\alpha x^\alpha t^{w\cdot\alpha},
 \end{equation}
where  $w\cdot\alpha$ is the usual dot product.
(To compare this to~\cite[\sect\ 15.8]{Eisenbud}, let $w=-\lambda$.) 
Let \defcolor{$w(f)$} be the minimum value of $w\cdot\alpha$ for $c_\alpha\neq 0$.
We define
 \begin{equation}\label{Eq:f_t}
     \defcolor{f_t}\vcentcolon=(t.f)t^{-w(f)}=f_w+ t\,g,
 \end{equation}
where the \demph{initial form} \defcolor{$f_w$} of $f$ is the sum of its terms $c_\alpha x^\alpha$ where
$w\cdot\alpha=w(f)$, and $g$ is a polynomial in the variables $t,x_0,\dotsc,x_n$. 

Let $X\subset\PP^n$ be a projective variety with ideal $I$.
Define $\calX^w\subset\PP^n\times\CC$ to be the Zariski closure of the family of translates of $X$, thus
 \[
  \defcolor{\calX^w} \vcentcolon=  \overline{\{(x,t)\in\PP^n\times\CC^\times :  x\in t.X\}}\subset\PP^n\times\CC_t.
 \]
For $t\not=0$, we observe that $\calX^w_t=t.X$ and has ideal $\langle f_t:f\in I\rangle$.
The following result establishes the flatness of this family:

%%%%%%%%%%%%%%%%%%%%%%%%%%%%%%%%%%%%%%%%%%%%%%%%%%%%%%%%%%%%%%%%%%%%%%%%%%%%%%%%%
\begin{proposition}[{\cite[\theo\ 15.17]{Eisenbud}}] 
  The family $\calX^w\to\CC_t$ is flat.
  The fiber at $t=0$ is the scheme with ideal
 \[
    I_w\ =\ \langle f_w :  f\in I\rangle.   
 \]
\end{proposition}
%%%%%%%%%%%%%%%%%%%%%%%%%%%%%%%%%%%%%%%%%%%%%%%%%%%%%%%%%%%%%%%%%%%%%%%%%%%%%%%%%

The proof uses a Gr\"obner basis $\calG$ for $I$ with respect to a weighted term order $\leq$
with weight $-w$ so that $f_w$ consists of the $\leq$-leading terms for
$f$.\footnote{We use $-w$ because the leading form in the weighted term order $\leq_\omega$
     for $\omega\in\ZZ^{n+1}$ is the
    sum of terms with highest $\omega$-weight, which is opposite our convention from valuations.}

Suppose that the family $\calX^w$ has Gr\"obner basis $\defcolor{\calG_t}\vcentcolon=\{g_t :  g\in\calG\}$.
Then a Gr\"obner basis for $I_w$ is obtained by setting $t=0$ in $\calG_t$.
The scheme at $t=0$ may be neither reduced nor irreducible. 
If this scheme is a toric variety, then the weight degeneration $\calX^w\to\CC_t$ is a toric degeneration. \hfill$\diamond$  
\end{example}
%%%%%%%%%%%%%%%%%%%%%%%%%%%%%%%%%%%%%%%%%%%%%%%%%%%%%%%%%%%%%%%%%%%%%%%%%%%%%%%%%

%%%%%%%%%%%%%%%%%%%%%%%%%%%%%%%%%%%%%%%%%%%%%%%%%%%%%%%%%%%%%%%%%%%%%%%%%%%%%%%%%

\begin{remark}\label{Ref:toriccomplex:example}
 Homotopy algorithms using weight degenerations
  appearing in the literature include the homotopy for solving the Kuramoto equations~\cite{CD_Kuramoto}
  and the Gr\"obner homotopy~\cite{HSS}.
  In these examples, $I_w$ is a square-free monomial ideal so that the special fiber is a union of linear spaces.
  Such degenerations can be handled by Algorithm~\ref{Alg:TDA}, see  Remark~\ref{Ref:ToricComplex}.
  \hfill$\diamond$
\end{remark}
%%%%%%%%%%%%%%%%%%%%%%%%%%%%%%%%%%%%%%%%%%%%%%%%%%%%%%%%%%%%%%%%%%%%%%%%%%%%%%%%%

%%%%%%%%%%%%%%%%%%%%%%%%%%%%%%%%%%%%%%%%%%%%%%%%%%%%%%%%%%%%%%%%%%%%%%%%%%%%%%%%%
\begin{example}\label{Ex:quasi-symmstry}
Algebraic statistics gives examples of toric degenerations \cite{KMS} which do not come from a weight degeneration.
 Let $G$ be a graph with vertex set $\defcolor{[m]}\vcentcolon= \{1,\dotsc,m\}$ and edge set
 $\defcolor{E}\subset\binom{[m]}{2}$.
 For each $i\in[m]$, let $a_i$ be a parameter.
 For each $\{i,j\}\in E$, let $x_{ij}=x_{ji}$ and define $p_{ij}$ and $p_{ji}$ via the formula
 \[
    \defcolor{p_{rs}}\vcentcolon=  x_{rs}(1+a_r-ta_s).
 \]
 These polynomials give a map $p\colon\CC^{|E|}\times\CC^m\times\CC_t\to \PP^{2|E|-1}\times\CC_t$ whose image is the family
 \defcolor{$\QS$} of \demph{quasi-symmetry models}.
 This family contains two known quasi-symmetry models, the Pearsonian quasi-symmetry model at $t=1$  and the
 toric quasi-symmetry model at $t=0$. 
 Polynomials associated to cycles in $G$ generate the ideal of the family
 $\QS$.
 In the proof of this fact, one step is to show that this family is flat.

 The family of quasi-symmetry models when $G$ is a $3$-cycle is  the family of
 hypersurfaces defined by the cubic
 \begin{multline}\label{eq:P}
   \quad\defcolor{P}\vcentcolon= (1+t+t^2)(p_{12}p_{23}p_{31} - p_{21}p_{32}p_{13})\, + \\
     t(p_{12}p_{23}p_{13}+p_{12}p_{32}p_{31}+p_{21}p_{23}p_{31}-
     p_{12}p_{32}p_{13}-p_{21}p_{23}p_{13}-p_{21}p_{32}p_{31}).\quad
  \end{multline}
 The fiber  $\QS_0$ at $t=0$ is the toric variety defined by the binomial $p_{12}p_{23}p_{31} - p_{21}p_{32}p_{13}$.

 The family of quasi-symmetry models $\QS$ for a graph is typically not a weight degeneration.
 In particular, the family defined in Equation \eqref{eq:P} is not a weight degeneration.
 Indeed, in each of the eight terms of $P$, exactly one of $p_{ij}$ or $p_{ji}$ occurs, so the terms
 correspond to the vertices of a cube.
 For any weight $w$, $P_w$ consists of the sum of terms identified with some face of the cube.
 Since the polynomial defining  $\QS_0$ corresponds to a diagonal of the cube, it is not of the form $P_w$, for
 any $w$. \hfill$\diamond$
\end{example}
%%%%%%%%%%%%%%%%%%%%%%%%%%%%%%%%%%%%%%%%%%%%%%%%%%%%%%%%%%%%%%%%%%%%%%%%%%%%%%%%%

%%%%%%%%%%%%%%%%%%%%%%%%%%%%%%%%%%%%%%%%%%%%%%%%%%%%%%%%%%%%%%%%%%%%%%%%%%%%%%%%%
\begin{example}\label{Ex:BS_TDA}
We present an example of a weight degeneration and use it to illustrate Algorithm~\ref{Alg:TDA}.
Let $\defcolor{X}\subset\PP^7$ be the closure of the image of the map $\defcolor{\varphi}\colon\CC^3\to\PP^7$ given by
\[
    \varphi(x,y,z)\ =\ [1, x, y, z, xz, yz, x(xz+y), y(xz+y)].
\]
This subvariety has degree six and its ideal \defcolor{$I$} has nine generators:
 \begin{equation*}%\label{Eq:GeneratorsI_B}
    \begin{array}{c}
    x_1x_3-x_0x_4,\, \                             %a
    x_2x_3-x_0x_5,\,\                              %b
    x_1x_2-x_0x_6+x_1x_4,\,\                       %c
    x_2^2-x_0x_7+x_3x_6-x_4^2,\, \vspace{1pt}\\  %d
    x_2x_6-x_1x_7,\,\         %1
    x_2x_5-x_3x_7+x_4x_5,\,\   %2 
    x_1x_5-x_3x_6+x_4^2,\,     %3 
    x_2x_4-x_1x_5,\,\         %4
    x_5x_6-x_4x_7.          %5
   \end{array}
 \end{equation*}
 Let $w=(-2,-1,-1,-1,0,0,0,0)$. 
We use Equation~\ref{Eq:f_t} to compute the ideal of $\calX^w$.
 The following thirteen polynomials form a Gr\"obner basis \defcolor{$\calG_t$} for $\calX^w$ with respect to
 the weighted term order $\leq_{-w}$:
  \begin{equation*}%\label{Eq:Gt}                              %  w-weight of lowest term
    \begin{array}{c}
    \underline{x_1x_3-x_0x_4}\,,\:\                           %2
    \underline{x_2x_3-x_0x_5}\,,\:\                           %2
    \underline{x_1x_2-x_0x_6}+tx_1x_4\,,\:\                    %2
    \underline{x_2^2-x_0x_7}+tx_3x_6-t^2x_4^2\,,\vspace{3pt}\\ %2
    \underline{x_2x_6-x_1x_7}\,,\:\                           %3
    \underline{x_2x_5-x_3x_7}+tx_4x_5\,,\:\                   %3
    \underline{x_1x_5-x_3x_6}+tx_4^2\,,\:\                    %3
    \underline{x_2x_4-x_1x_5}\,,\:\                          %3
    \underline{x_5x_6-x_4x_7}\,,       \vspace{3pt}          %4
\\          
    \underline{x_0x_6^2-x_1^2x_7}-tx_1x_4x_6\,,\:\            %4
    \underline{x_0x_5^2-x_3^2x_7}+tx_3x_4x_5\,,\vspace{3pt}\\ %4
    \underline{x_0x_4x_5-x_3^2x_6}+tx_3x_4^2\,,\:\            %4
    \underline{x_3x_6^2-x_1x_4x_7}-tx_4^2x_6.                 %5
  \end{array}
\end{equation*}
The leading terms with respect to $\leq_{-w}$ are underlined, and these binomials generate the
ideal \defcolor{$I_w$}.  This ideal is the toric ideal of the image of the map
$\varphi_\calA(x,y,z)=[1, x, y, z, xz, yz, xy, y^2]$ given by the lowest order monomials in $\varphi$.  
For the toric ideal statement, observe that if we set $(x_0,x_1,x_2,x_3)=(1,x,y,z)$,
then the first four underlined binomials in $\calG_t$
express $x_4,\dotsc,x_7$ as the monomials in $x,y,z$ appearing in $\varphi_\calA$.
The exponent vectors of $\varphi_\calA$ are the columns of the matrix $\calA$ in Figure~\ref{F:ANP}.
%%%%%%%%%%%%%%%%%%%%%%%%%%%%%%%%%%%%%%%%%%%%%%%%%%%%%%%%%%%%%%%%%%%%%%%%%%%%%%%%%
\begin{figure}[htb]
  \centering
   $ \calA\ =\ \left(\begin{matrix}
        0&1&0&0&1&0&1&0\\0&0&1&0&0&1&1&2\\0&0&0&1&1&1&0&0
      \end{matrix}\right)$
      \qquad\quad
      \raisebox{-37.5pt}{\includegraphics{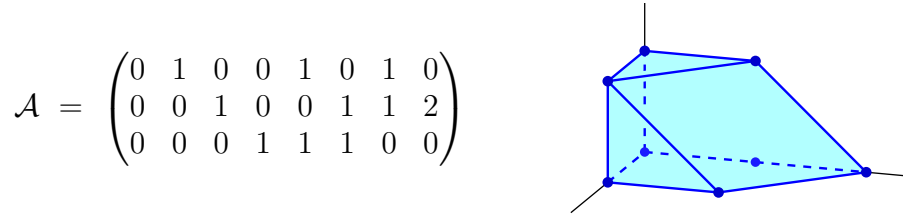}}
    
      \caption{The weight vectors for the toric Kodaira map $\varphi_\calA$ are the columns of matrix $\calA$.
        The Newton polytope is the convex hull of these vectors.}
    \label{F:ANP}
\end{figure}
%%%%%%%%%%%%%%%%%%%%%%%%%%%%%%%%%%%%%%%%%%%%%%%%%%%%%%%%%%%%%%%%%%%%%%%%%%%%%%%%%

Let $\defcolor{L}\subset\PP^7$ be the linear subspace of codimension three whose defining equations are
$\ell_i = \sum c_{ij} x_j$ for $i=1,2,3$, where $C=(c_{ij})$ is the
$3\times 8$ matrix 
 \begin{equation*}%\label{Eq:CoeffsLinSection}
    C\ =\ \left(\begin{matrix}
        1&1&1&1&1&1&1&1\\1&-2&3&-4&5&-6&7&-8\\2&3&5&7&11&13&17&19
      \end{matrix}\right)\,.
 \end{equation*}
 The subspace $L$ meets both $\overline{\varphi(\CC^3)}=\calX_1$ and
 $\overline{\varphi_\calA((\CC^\times)^3)}=\calX_0$ transversally in six points.

We follow the steps of Algorithm~\ref{Alg:TDA} to compute $\calX_1\cap L$.
We first compute the sparse system $G_\calA$ in Step \ref{TDA0} of Algorithm~\ref{Alg:TDA} to arrive at the system
\begin{align*}%\label{Eq:BS_Sparse_system}
  1+x+y+z+xz+yz+xy+y^2&=0\nonumber\\
  1-2x+3y-4z+5xz-6yz+7xy-8y^2&=0\\
  2+3x+5y+7z+11xz+13yz+17xy+19y^2&=0.\nonumber
\end{align*}
 In Step \ref{TDA1} of Algorithm~\ref{Alg:TDA}, we compute the six solutions of the system $G_\calA$, 
 one of which is  $\zeta=(-1.33613,1.51406,-1.22871)$. 
 The image $\varphi_\calA(\zeta)$ in $\PP^7$ is
 \[
   [1 ,\, -1.33613 ,\, 1.51406 ,\, -1.22871 ,\, 1.64171 ,\, -1.86035 ,\, -2.02298 ,\, 2.29239].
 \]
 In Step~\ref{TDA2} of Algorithm~\ref{Alg:TDA}, we compute the images of
 these six solutions under $\varphi_\calA$, which forms the points of $\calX_0\cap L$.
 Therefore, the images of these points are the solutions to the start system for the linear section
 homotopy given by $H(x,t)=(\calG_t,L)$.
 In Step~\ref{TDA3}, these solutions are followed from $t=0$ to $t=1$, computing the
 six points of the linear section $\calX_1\cap L$.
 One point of $\calX_1\cap L$ is 
 \[
   [1,\, -0.689522,\, 0.928435,\, -1.35986,\, 0.937652,\, -1.26254,\, -1.28671,\, 1.73254].
   \eqno{\diamond}
 \]
\end{example}

%%%%%%%%%%%%%%%%%%%%%%%%%%%%%%%%%%%%%%%%%%%%%%%%%%%%%%%%%%%%%%%%%%%%%%%%%%%%%%%%%%%%%%%%%%%%
\section{Khovanskii bases and the Khovanskii homotopy}\label{S:Khovanskii}

Let $X$ be a complex variety and $V\subset\CC(X)$ be a finite-dimensional complex vector space of rational functions
on $X$. 
The closure of the image of $X$ under the Kodaira map $\varphi_V\colon X\dashrightarrow \PP(V^*)$ has homogeneous coordinate
ring $R(V)$ generated by $V$.  
When this ring has a finite Khovanskii basis contained in $V$, Anderson's toric degeneration embeds in $\PP(V^*)$ as a
weight degeneration.
We use this degeneration in the Khovanskii homotopy algorithm (Algorithm \ref{Alg:Khovanskii}) to compute a linear section $\varphi_V(X)\cap L$.

We review the theory of Newton-Okounkov bodies and  Khovanskii bases and then describe how to produce an embedding of
Anderson's toric 
degeneration into $\PP(V^*)$ when the Khovanskii basis is a subset of $V$.
We also show how to compute a Kodaira map of the toric special fiber.
With the embedding and toric Kodaira map, Algorithm~\ref{Alg:TDA} becomes an effective method to compute linear sections.
In Section~\ref{S:WPS}, we explain how to modify this method for the general case when the Khovanskii basis is not a
subset of $V$.

%%%%%%%%%%%%%%%%%%%%%%%%%%%%%%%%%%%%%%%%%%%%%%%%%%%%%%%%%%%%%%%%%%%%%%%%%%%%%%%%%
\subsection{Valuations, Khovanskii bases, and Newton-Okounkov bodies}%\label{sec:valuations:Khovanskii:NO}

We recall the key definitions and properties of Khovanskii bases from \cite{KM_tropical}.
Suppose that $X$ is a $d$-dimensional complex variety with function field $\CC(X)$.
Let $\succ$ be a total order on $\ZZ^d$ so that $\ZZ^d$ is an ordered abelian group.
A \demph{$\ZZ^d$-valuation} on $\CC(X)$ is a surjective group homomorphism $\nu\colon\CC(X)^\times\mapsto\ZZ^d$ satisfying
the property that for all $f,g \in\CC(X)$ and $c\in \CC^\times$, 
 \begin{center}
    $\nu(f+g) \succeq \min\{\nu(f), \nu(g)\}$ \ and \  $\nu(c) = 0$.
 \end{center}
By convention, $\nu(0)=\infty$, $\infty\succeq \alpha$, and $\alpha+\infty=\infty$ for all $\alpha\in\ZZ^d$.
 Since $\dim X=d$, $\nu$ is a surjection, and $\CC$ is algebraically closed,
 it follows that
if $f,g\in\CC(X)^\times$ with $\nu(f)=\nu(g)$, then there is a unique $c\in\CC^\times$ with $\nu(f{-}cg)\succ\nu(f)$.

Let $V$ be a finite-dimensional complex vector subspace of $\CC(X)$.
We assume that the image of $V^\times$ under $\nu$ generates $\ZZ^d$ (see Remark~\ref{R:nuGenerates}).
We write $\defcolor{R(V)}$ for the graded ring
$\bigoplus_{k\ge 0} V^ks^k$, where
$\defcolor{V^k}\subset\CC(X)$ is the subspace spanned by all $k$-fold products of elements in $V$ and \defcolor{$s$} is 
a formal variable recording the grading.
A nonzero element $f\in R(V)^\times$ is the sum of its homogeneous components,
 \begin{equation*}
   %\label{Eq:inHommSum}
   f = f_k s^k + \dotsb + f_1 s + f_0,
 \end{equation*}
where $f_k\neq 0$ and $f_i\in V^i$ for all $i$.
We extend the valuation $\nu$ to $R(V)$ by defining $\defcolor{\nu(f)}\vcentcolon= (\nu(f_k), k)\in\ZZ^d\oplus \NN$.
We also extend $\succeq$ to $\ZZ^d\oplus\NN$, where $(\alpha,k)\succ(\beta,l)$ if $k<l$ or else $k=l$ and $\alpha\succ\beta$ in the order on
$\ZZ^d$. 
The direction of the inequality in $k<l$ is chosen to be consistent with $\nu(f)= (\nu(f_k), k)$ defining a valuation.

We write $\defcolor{S(V,\nu)}$ for the image $\{\nu(f) :  f \in R(V)^\times\}$ of  $R(V)^\times$ under $\nu$.
This is a submonoid of $\ZZ^d\oplus\NN$.
The closure of the convex hull of $S(V,\nu)$ in $\RR^d \times \RR$ is the cone $\defcolor{\text{cone}(V)}$.
Its base $\defcolor{\NO_V}\vcentcolon=  \text{cone}(V) \cap (\RR^d \times \{1\})$ is the \demph{Newton-Okounkov body} of
$V$. 
The Newton-Okounkov body carries a considerable amount of information about $R(V)$, see~\cite{KaKhb,LaMu}.
For example, the number of solutions to System~\eqref{Eq:generalSystem} where $f_1,\dotsc,f_d\in V$ are general 
(in this case, we say that System~\eqref{Eq:generalSystem} is \demph{drawn} from $V$) is
  the normalized volume of $\NO_V$.
  
A \demph{Khovanskii basis}~\cite{KM_tropical} for $V$ is a linearly independent set $\calB\subset R(V)$
whose image under $\nu$ generates $S(V,\nu)$.
We assume that the elements of $\calB$ are homogeneous so that for $b\in\calB$ with $\nu(b)=(\alpha,k)$, $b\in V^ks^k$.
Necessarily, $\calB$ generates $R(V)$ and $\calB\cap Vs$ is a basis for $Vs$.
We observe that $S(V,\nu)$ is finitely generated if and only if $V$ has a finite Khovanskii basis.
When $V$ has a finite Khovanskii basis, Anderson~\cite{Anderson} shows that
$\NO_V$ is a rational polytope and that there exists a flat degeneration $\calX \to \CC_t$ of
$\calX_1\simeq\Proj(R(V))$ to the toric variety
$\calX_0\simeq\Proj(\CC[S(V,\nu)])$.

The valuation $\nu$ on $R(V)$ induces a filtration on $R(V)$ by finite-dimensional subspaces indexed by elements
$(\alpha,k)\in S(V,\nu)$. 
We let
 \begin{equation*}%\label{Eq:Filtration}
 \begin{array}{rcl}
   \defcolor{R(V)_{(\alpha,k)}}&\vcentcolon= &\{f\in R(V) :  \nu(f)\succeq(\alpha,k)\},
        \mbox{  and}\vspace{4pt}\\ 
  \defcolor{R(V)^+_{(\alpha,k)}}&\vcentcolon= &\{f\in R(V) :  \nu(f)\succ(\alpha,k)\}.
 \end{array}
 \end{equation*}
Since $(\alpha,k)\in S(V,\nu)$, these subspaces satisfy $R(V)_{(\alpha,k)}/R(V)^+_{(\alpha,k)}\simeq\CC$.
Anderson's flat degeneration comes from the degeneration of the filtered algebra $R(V)$ to its associated graded algebra 
\[
  \defcolor{\gr R(V)} \vcentcolon=  \bigoplus_{(\alpha,k)\in S(V,\nu)}R(V)_{(\alpha,k)}/R(V)^+_{(\alpha,k)}\ \simeq\
  \CC[S(V,\nu)].
\]
The toric fiber $\calX_0$ of Anderson's degeneration is
$\Proj(\gr R(V))$, and the isomorphism $\calX_0\simeq\Proj(\gr R(V))$ uses the isomorphism
$\gr R(V)\simeq \CC[S(V,\nu)]$.

Kaveh and Manon give a method to compute a finite Khovanskii basis for $V$ with respect
to a valuation $\nu$~\cite[\algo\ 2.18]{KM_tropical}.
We take a finite Khovanskii basis as an input to our algorithms.

%%%%%%%%%%%%%%%%%%%%%%%%%%%%%%%%%%%%%%%%%%%%%%%%%%%%%%%%%%%%%%%%%%%%%%%%%%%%%%%%%
\subsection{The Kodaira map and embedding the degeneration}\label{sec:embedding}

To use Anderson's toric degeneration $\calX$
in Algorithm~\ref{Alg:TDA}, $\calX$ must be embedded in a projective space.
Suppose that we are given a finite Khovanskii basis $\calB$ for $V$ such that $\calB\subset Vs$.
Therefore, $\calB$ is a basis for $Vs$, by definition.

Let $X^\circ\subset X$ be the open subset of points of $X$ where no function from $V$ has a pole, and some function in $V$
is nonzero.
Evaluation of functions from $V$ at a point $z\in X^\circ$ gives a nonzero linear map
$\defcolor{\ev_z(f)}\vcentcolon= f(z)$ on $V$.
Therefore, $\ev_z$ is a point in the projective space $\PP(V^*)$, where \defcolor{$V^*$} is the space of linear functions
$V\to\CC$. 
Thus the map $z\mapsto\ev_z$ induces a map $X^\circ\to\PP(V^*)$, which is called the \demph{rational Kodaira map}
$\varphi_V \vcentcolon X\dasharrow\PP(V^*)$. 
If we write $\calB=\{b_0s,\dots,b_ns\}$, then
a Kodaira map can be explicitly written as
$\defcolor{\varphi_\calB}\colon z\in X^\circ\mapsto[b_0(z),\dotsc,b_n(z)]\in\PP^n\simeq\PP(V^*)$.

%%%%%%%%%%%%%%%%%%%%%%%%%%%%%%%%%%%%%%%%%%%%%%%%%%%%%%%%%%%%%%%%%%%%%%%%%%%%%%%%%
\begin{remark}\label{R:nuGenerates}
  Our algorithms compute the points of $\varphi_\calB(X^\circ)\cap L$.
  Given these points, the solutions to System~\eqref{Eq:generalSystem} on $X^\circ$ are their pull backs along $\varphi_\calB$.
  When the Kodaira map is not injective, we follow Am\'endola and Rodriguez~\cite{AmendolaRodriguez} and note that these 
  pull backs may be computed from the linear section and the points in a single general fiber of the
  Kodaira map. 

  Consequently, we assume that the Kodaira map is an injection and replace $X$ by its birational
  copy $\Proj(R(V))$, which is the closure of $\varphi_\calB(X^\circ)$ in $\PP^n$.
  In this case, $X=X^\circ$, $V$ generates the function field $\CC(X)$ of $X$, and the image of $V^\times$ under 
  $\nu$ generates $\ZZ^d$.
  Thus the assumption
  that $X=\Proj(R(V))$ implies that the image of $V^\times$ under $\nu$ generates $\ZZ^d$.\hfill$\diamond$ 
\end{remark}
%%%%%%%%%%%%%%%%%%%%%%%%%%%%%%%%%%%%%%%%%%%%%%%%%%%%%%%%%%%%%%%%%%%%%%%%%%%%%%%%%

We recall the embedding of Anderson's toric degeneration $\calX$ into
$\PP(V^*)$~\cite[\sect\ 2.2]{KM_tropical}.
We let $\defcolor{\calA}\vcentcolon=\nu(\calB)$ be the $({d+1})\times (n+1)$ 
matrix whose ${i}^{\rm th}$ column is $\nu(b_{i-1}s)$
for the Khovanskii basis $\calB=\{b_0s,\dots,b_ns\}\subset Vs$.
We note that the last row of $\calA$ is $(1,\dotsc,1)$.
We define a partial order \defcolor{$>_{\calA}$} on $\ZZ^{n+1}$ where $\beta>_{\calA}\alpha$
if $\calA\alpha  \succ \calA\beta$ in $\ZZ^{d+1}$.
The initial form $\defcolor{\ini_{\calA}(f)}$ of a polynomial $f$ with respect to $>_{\calA}$ is the sum of all
terms $c_\alpha x^\alpha$ which minimize $\calA\alpha$.

The ideal \defcolor{$I_\calB$} of $X=\varphi_\calB(X)$ is the kernel of the map $\CC[x_0,\dots,x_n]\to R(V)$ which takes
$x_i$ to $b_is$. 
We define \defcolor{$\ini_{\calA}(I_\calB)$} to be the ideal generated by $\ini_{\calA}(f)$ for $f\in I_\calB$.
Anderson~\cite[\lemm\ 8]{Anderson} shows that there exists $w\in\ZZ^{d+1}$ such that if $\leq_{-w\calA}$ is the
weighted term order on $\CC[x_0,\dots,x_n]$ induced by $-w\calA$, then the leading term ideal
\defcolor{$\lt_{-w\calA}(I_\calB)$}  of $I_\calB$ equals $\ini_{\calA}(I_\calB)$.
Let $w$ be such a weight vector and \defcolor{$\calG$} denote a Gr\"obner basis for $I_\calB$  with respect to 
a total order induced from the term
order $\leq_{-w\calA}$. 
The leading forms of elements of $\calG$ with respect to $\leq_{-w\calA}$ generate
$\ini_{\calA}(I_\calB)$.

Let $g  = \sum _\alpha c_\alpha x^\alpha$ be a polynomial in $\calG$, and define
$\defcolor{w(g)}\vcentcolon=  \min \{ w \calA \alpha : c_\alpha \ne 0\}$.
Using Formula~\eqref{Eq:f_t} (with $w\calA$ in place of $w$), we construct
 \begin{equation}\label{eq:gt}
  g_t = \sum_\alpha c_\alpha x^\alpha t^{w \calA \alpha-w(g)}.
 \end{equation}
Let $\defcolor{\calG_t}\vcentcolon= \{g_t :  g\in\calG\}$.
At $t=0$, $\calG_0$ generates $\ini_{\calA}(I_\calB)$ and at $t=1$, $\calG_1=\calG$ generates $I_\calB$.  

Finally, we define \defcolor{$I_\calA$} to be the kernel of the map $\CC[x_0, \dots, x_n]\to\gr R(V)$  which takes $x_i$ to
$\overline{b_is}\in R(V)_{(\nu(b_i),1)}/R(V)^+_{(\nu(b_i),1)}$. 
We note that $I_\calA$ is a toric ideal, and by~\cite[\theo\ 2.17]{KM_tropical}, $I_\calA=\ini_{\calA}(I_\calB)$.
Thus the toric weight degeneration can be embedded into $\PP^n\simeq\PP(V^*)$.

%%%%%%%%%%%%%%%%%%%%%%%%%%%%%%%%%%%%%%%%%%%%%%%%%%%%%%%%%%%%%%%%%%%%%%%%%%%%%%%%%
\begin{proposition}[{\cite[Theorem 1]{Anderson}}]%\label{T:KHA}
  Let $X$ be a variety and $V\subset \CC(X)$ a finite-dimensional space of functions which has a finite Khovanskii basis
  $\calB\subset Vs$. 
  Then the family $\calX \to \CC_t$ defined by $\calG_t$
  is flat and embeds into $\PP^n$ as the weight degeneration of $\calX_1  = \Proj(R(V))= \varphi_\calB(X)$ induced by $w\calA$.
  In particular, $\calX_0\simeq\Proj(\CC[S(V,\nu)])$ and $\calX$ is a toric degeneration.
\end{proposition}
%%%%%%%%%%%%%%%%%%%%%%%%%%%%%%%%%%%%%%%%%%%%%%%%%%%%%%%%%%%%%%%%%%%%%%%%%%%%%%%%%
%
We now discuss the relationship between $I_\calA$ and $I_\calB$.
For $u\in\NN^{n+1}$, we write \defcolor{$\calB^u$} for the product $\prod (b_i s)^{u_i}$ of elements in
the Khovanskii basis.
Since $\nu(\calB^u)=\calA u$, when $\calA u = \calA v$ for some $u,v\in\NN^{n+1}$,
$\nu(\calB^u)=\nu(\calB^v)$ and there is a unique $\defcolor{c}\in\CC^\times$ such that 
\[
   \calA u = \calA v \prec \nu(\calB^u-c\calB^v)\quad\text{and}\quad \calB^u-c\calB^v\in R(V)^+_{\calA u}.
\]
Since the last row of $\calA$ is $(1,\dotsc,1)$, both $\calB^u$ and $c\calB^v \in V^k s^k$ for some $k$ and 
their difference is homogeneous.

The subduction algorithm~\cite[\algo\ 2.11]{KM_tropical} rewrites this difference as a homogeneous polynomial of degree $k$ 
in the elements of the Khovanskii basis,
\[
  \calB^u-c\calB^v\ =\ h(b_0s, b_1s,\dotsc,b_ns).
\]
In particular, $g\vcentcolon= x^u-cx^v-h(x_0,\dotsc,x_n)\in I_\calB$ with initial form $x^u-cx^v\in I_\calA$.
Applying Formula~\eqref{eq:gt}, we have that
\[
  g_t =  x^u-cx^v - t^r h_t,
\]
where $r=w(h)-w(g)>0$.

%%%%%%%%%%%%%%%%%%%%%%%%%%%%%%%%%%%%%%%%%%%%%%%%%%%%%%%%%%%%%%%%%%%%%%%%%%%%%%%%%
\begin{remark}\label{R:calBKodaira}
 We recall that the torus $\TT=(\CC^\times)^{n+1}/\Delta\CC^\times\simeq(\CC^\times)^n$ is the set of points in
 $\PP^n$ with nonzero coordinates. 
 A Kodaira map for the toric fiber $\calX_0$ has the form $\varphi_{p,\calA}$, as in 
 Formula~\eqref{Eq:translatedToricKodaira}, for any $p\in\TT\cap\calX_0$.
 We provide a construction of such a point. 
  
 Let $x^u-cx^v\in I_\calA$.
 Then $\calA u=\calA v$, so that $u-v\in \ker(\calA)$.
 Restricting this binomial to $\calX_0\cap\TT$ results in the equation $c=x^{u-v}$.
 The constant $c$ depends upon $u-v\in\ker(\calA)$, and we write \defcolor{$c_{u-v}$} for $c$.
 Thus a point $p\in\calX_0\cap\TT$ satisfies equations of the form
 \begin{equation*}%\label{eq:toricpoints}
     c_u = p^u
 \end{equation*}
 for $u\in\ker(\calA)$.
 While every $u\in\ker(\calA)$ gives such an equation,
 an independent set of equations is given by 
 a basis $u_1,\dotsc,u_{n-d}$ for $\ker(\calA)$.
 The corresponding equations,
$c_{u_i}=p^{u_i}$ for $i=1,\dotsc,n{-}d$, define $\calX_0\cap\TT$ as a subvariety of $\TT$.

 To obtain a point of $\calX_0\cap\TT$, 
 we construct $d$ additional equations to these $n{-}d$ equations, 
 as follows: 
 Since $\defcolor{\bbI}\vcentcolon=(1,\dotsc,1)$ is a row of $\calA$, 
 $\ker(\calA)\subset\ker(\bbI)$, which is a rank $n$ sublattice of $\ZZ^{n+1}$.
 Let $v_1,\dotsc,v_d\in\ker(\bbI)$ be vectors such that
$u_1,\dotsc,u_{n-d},v_1,\dotsc,v_d$ are
 independent.
 Choose nonzero constants $c_{v_1},\dotsc,c_{v_d}\in\CC^\times$ and consider the
 system of binomials 
 \[
   c_{u_i}-p^{u_i}=0=c_{v_j}-p^{v_j}\quad\text{for}\quad i=1,\dotsc,n{-}d \ \mbox{ and }\ j=1,\dotsc,d.
 \]
 This system defines a finite set of points $p\in\calX_0\cap\TT$.
 An algorithm for solving such a system of binomials is given in~\cite[\lemm\ 3.2]{HS95}, which 
 involves computing the Smith normal form of the matrix whose columns are
 $u_1,\dotsc,u_{n-d},v_1,\dotsc,v_d$.
 We observe that only one solution is needed to obtain a Kodaira map.  \hfill$\diamond$
\end{remark}
%%%%%%%%%%%%%%%%%%%%%%%%%%%%%%%%%%%%%%%%%%%%%%%%%%%%%%%%%%%%%%%%%%%%%%%%%%%%%%%%%

%%%%%%%%%%%%%%%%%%%%%%%%%%%%%%%%%%%%%%%%%%%%%%%%%%%%%%%%%%%%%%%%%%%%%%%%%%%%%%%%%
\subsection{Khovanskii homotopy}
The procedure described in Section~\ref{sec:embedding}, combined with the toric two-step homotopy algorithm,
Algorithm~\ref{Alg:TDA}, forms the Khovanskii homotopy algorithm for computing the points of a linear section
$\varphi_V(X)\cap L$.

%%%%%%%%%%%%%%%%%%%%%%%%%%%%%%%%%%%%%%%%%%%%%%%%%%%%%%%%%%%%%%%%%%%%%%%%%%%%%%%%%
\begin{algorithm}[Khovanskii homotopy algorithm]\
  \label{Alg:Khovanskii}

  {\bf Input:}   A finite-dimensional subspace $V\subset\CC(X)$ for a variety $X=\Proj(R(V))$ of

  \mbox{{\color{white}{\bf Input:}}}
  dimension $d$, a finite Khovanskii basis $\calB\subset Vs$ for $V$, and a general linear 

  \mbox{{\color{white}{\bf Input:}}}
  subspace $L\subset\PP^n$ of codimension $d$.
  
{\bf Output:} All points in the linear section
   $\varphi_V(X)\cap L\subset\PP(V^*)$.

  {\bf Do:}
\begin{enumerate}[(\roman*)]
\item Compute $I_\calB=\ker(\CC[x_0,\dots,x_n]\rightarrow R(V))$ where $x_i\mapsto b_is$. \label{Alg:Khovanskii:Step2}
\item Compute a weight vector $w$ using \cite[\lemm\ 2]{Anderson} so that
  $\lt_{-w\calA}(I_\calB)=\ini_{\calA}(I_\calB)$,
  where $\calA$ is the matrix of valuations of $\calB$. \label{Alg:Khovanskii:Step3}

\item Compute a Gr\"obner basis $\calG$ for $I_\calB$ using the weight $-w\calA$. \label{Alg:Khovanskii:Step4}

\item Construct the homotopy $\calG_t$ using Formula~\eqref{eq:gt}.\label{Alg:Khovanskii:Step5}
  
\item Construct the Kodaira map $\varphi_{p,\calA}$ for
  $\calX_0$ by following Remark~\ref{R:calBKodaira}.\label{Alg:Khovanskii:Step6} 

\item Return the output $\varphi_V(X)\cap L$ of  Algorithm~\ref{Alg:TDA} with input $\calG_t$ and
  $L$.%\label{Alg:Khovanskii:Step8} 
\end{enumerate}

\end{algorithm}
%%%%%%%%%%%%%%%%%%%%%%%%%%%%%%%%%%%%%%%%%%%%%%%%%%%%%%%%%%%%%%%%%%%%%%%%%%%%%%%%%

%%%%%%%%%%%%%%%%%%%%%%%%%%%%%%%%%%%%%%%%%%%%%%%%%%%%%%%%%%%%%%%%%%%%%%%%%%%%%%%%%
\begin{theorem}%\label{T:KHA2}
  Algorithm \ref{Alg:Khovanskii} is an optimal homotopy 
  algorithm for computing all points of $\varphi_V(X)\cap L$.
\end{theorem}
%%%%%%%%%%%%%%%%%%%%%%%%%%%%%%%%%%%%%%%%%%%%%%%%%%%%%%%%%%%%%%%%%%%%%%%%%%%%%%%%%
The correctness of Algorithm~\ref{Alg:Khovanskii} follows from the discussion in Section~\ref{sec:embedding}.
%%%%%%%%%%%%%%%%%%%%%%%%%%%%%%%%%%%%%%%%%%%%%%%%%%%%%%%%%%%%%%%%%%%%%%%%%%%%%%%%%
\begin{remark}\label{R:precompute}
 In many cases, Algorithm~\ref{Alg:Khovanskii} is applied to systems of functions where a finite Khovanskii basis 
 is explicitly known from the theory (see Example \ref{Ex:BSextrinsic}).  
 In this case, we not only have the data for the finite Khovanskii basis $\calB$, but also
  some or all of the data for Steps \ref{Alg:Khovanskii:Step2}, \ref{Alg:Khovanskii:Step3}, and \ref{Alg:Khovanskii:Step4} of
  Algorithm~\ref{Alg:Khovanskii}.
\hfill$\diamond$
\end{remark}
%%%%%%%%%%%%%%%%%%%%%%%%%%%%%%%%%%%%%%%%%%%%%%%%%%%%%%%%%%%%%%%%%%%%%%%%%%%%%%%%%

%%%%%%%%%%%%%%%%%%%%%%%%%%%%%%%%%%%%%%%%%%%%%%%%%%%%%%%%%%%%%%%%%%%%%%%%%%%%%%%%%
\begin{example}\label{Ex:BSextrinsic}
  We illustrate Algorithm~\ref{Alg:Khovanskii} and Remark~\ref{R:precompute} on a continuation of Example \ref{Ex:BS_TDA}.
  In~\cite[Section 6.4]{Anderson}, Anderson considers a particular three-dimensional Bott-Samelson variety $X$ for
  $GL(3,\CC)$ and an ample line bundle $\calL$ on $X$. 
  In local coordinates $(x,y,z)$ for $X$, the vector space \defcolor{$V$} of global sections of $\calL$ has basis
  $\{1,x,y,z,xz,yz,x(xz+y),y(xz+y)\}$.

  Anderson uses a valuation $\nu$ induced by the monomial valuation on $\CC[x,y,z]$ defined by $\nu(f)=(a,b,c)$, where
  $x^ay^bz^c$ is the monomial of 
    $f$ that is minimal in the degree lexicographic order with $x>y>z$.
  The image   $\calB= \{1s,xs,ys,zs,xzs,yzs,x(xz+y)s,y(xz+y)s\}$ of this basis in $Vs$ forms a 
  Khovanskii basis for $V$.
  The corresponding matrix of valuations is  
  \[
       \calA = \nu(\calB) = \left(\begin{matrix} 0 & 1 & 0 & 0 & 1 & 0 & 1 & 0 \\ 
                            			0 & 0 & 1 & 0 & 0 & 1 & 1 & 2 \\
                            			0 & 0 & 0 & 1 & 1 & 1 & 0 & 0 \\
                                                1 & 1 & 1 & 1 & 1 & 1 & 1 & 1 \end{matrix}\right),
 \]
  which is the matrix of Figure~\ref{F:ANP} after appending the row $\bbI$ for the exponents of $s$.
  The Newton-Okounkov body of $V$ is also displayed in Figure~\ref{F:ANP}.

  Anderson provides the Khovanskii basis $\calB$ for Algorithm~\ref{Alg:Khovanskii}, and  Example~\ref{Ex:BS_TDA}
    gives the general linear section $L$.
  For Step \ref{Alg:Khovanskii:Step2}, generators of $I_\calB$ are the generators of $I$ in
  Example~\ref{Ex:BS_TDA}.
  The weight vector $w=(1,1,1,-2)$ suffices for
  Step \ref{Alg:Khovanskii:Step3}.  
  The vector $w\calA = (-2,-1,-1,-1,0,0,0,0)$ appears as the weight in Example~\ref{Ex:BS_TDA}.
  The computations in Steps \ref{Alg:Khovanskii:Step4} and \ref{Alg:Khovanskii:Step5} are supplied
  by the elements in $\calG_t$ in Example~\ref{Ex:BS_TDA}.
  Finally, for Step \ref{Alg:Khovanskii:Step6}, the toric Kodaira map
  $\varphi_\calA$ is also given in Example~\ref{Ex:BS_TDA}.  \hfill$\diamond$
\end{example}

%%%%%%%%%%%%%%%%%%%%%%%%%%%%%%%%%%%%%%%%%%%%%%%%%%%%%%%%%%%%%%%%%%%%%%%%%%%%%%%%%
 \section{The Khovanskii homotopy for weighted projective space}\label{S:WPS}

When a Khovanskii basis $\calB$ for $V$ contains elements of degree greater than 1, Anderson's
toric degeneration naturally embeds into a weighted projective space~\cite{Anderson}.
We explain how to 
lift the degeneration to a toric degeneration in ordinary projective space and use the toric two-step homotopy
  (Algorithm~\ref{Alg:TDA}) to compute a linear section 
$\varphi_V(X)\cap L$ of the image of $X$ under the Kodaira map $\varphi_V\colon X\dasharrow\PP(V^*)$.

%%%%%%%%%%%%%%%%%%%%%%%%%%%%%%%%%%%%%%%%%%%%%%%%%%%%%%%%%%%%%%%%%%%%%%%%%%%%%%%%%
\subsection{Weighted projective spaces}\label{sec:weightedReview}
We recall the construction and some basic properties of weighted projective space, see~\cite{Dolgachev}.  
Suppose that $a=(a_0,\dots,a_{n+m})$ is a vector of mutually relatively prime positive integers.
The \demph{weighted projective space} \defcolor{$\PP^{n+m}_a$} is
$\Proj(\CC[x_0,\dots,x_{n+m}])$, where the grading on $\CC[x_0,\dots,x_{n+m}]$ is induced by setting the degree of $x_j$
to $a_j$.  
Equivalently, $\PP^{n+m}_a$ is the quotient of
$\CC^{n+m+1}\smallsetminus\{0\}$ by the 
$\CC^\times$-action where $t.(x_0,\dots,x_{n+m})=(t^{a_0}x_0,\dots,t^{a_{n+m}}x_{n+m})$, for $t\in\CC^\times$.
We may also construct $\PP^{n+m}_a$ as a quotient of $\PP^{n+m}$. 
To see this, let $\Delta\CC^\times\subset (\CC^\times)^{n+m+1}$ be the diagonal embedding of $\CC^\times$ and 
let $\defcolor{G_a}$ be the image of the following product of groups of roots of unity 
in the dense torus $(\CC^\times)^{n+m+1}/\Delta(\CC^\times)$ of $\PP^{n+m}$:
\[
  \mbox{Hom}\left(\prod_j \ZZ/a_j\ZZ,\CC^\times\right)
  = \prod_j  \mbox{Hom}\left(\ZZ/a_j\ZZ,\CC^\times\right)\subset
  (\CC^\times)^{n+m+1}.
\]
Thus $G_a$ acts faithfully on $\PP^{n+m}$.
As the $a_j$ are mutually relatively prime, $G_a$ is isomorphic to this product of groups of roots of unity.
Let $\defcolor{\pi}\colon\PP^{n+m}\rightarrow\PP^{n+m}_a$ be the quotient map by this $G_a$-action, which
is a finite map of degree $|G_a|=\prod a_j$.

The weighted projective spaces that appear in the Khovanskii homotopy have the following special form:
Let $W=\bigoplus_{k\geq 1} W_k$ be a finite-dimensional positively-graded vector space
with $\dim W_1=n{+}1\geq 1$.
Let $t\in\CC^\times$ act on $W_k$ as multiplication by $t^{-k}$, which gives a $\CC^\times$-action on $W$.
Identifying the dual space $W^*$ with $\bigoplus_{k} W_k^*$, in the dual action, $t\in\CC^\times$ acts on $W^*_k$ as
multiplication by $t^k$.
Then the quotient of $W^*\smallsetminus\{0\}$ by $\CC^\times$ is a weighted projective space.

We explicitly describe this weighted projective space.
Suppose that $\dim W=n{+}m{+}1$, and let $a=(a_0,\dotsc,a_{n+m})$ be a vector in which each
$k\in\NN$ occurs $\dim W_k$ times. 
Then $(W^*\smallsetminus\{0\})/\CC^\times$ is isomorphic to $\PP^{n+m}_a$, and we write \defcolor{$\PP_a(W^*)$} for this
quotient. 
The isomorphism depends upon the choice of an ordered basis for $W^*$ which is a union of bases
for each nontrivial summand $W^*_k$ such that $a_j=k$ when the $j$th basis element lies in $W^*_k$.
This choice of basis identifies $W^*$ with $\CC^{n+m+1}$, and allows us to define an action of $G_a$ on the projective
  space $\PP^{n+m}$ with quotient map $\pi\colon\PP^{n+m}\to\PP_a(W^*)$ as in the first paragraph above.
We remark that there is no natural identification of $\PP(W^*)$ with $\PP^{n+m}$ that is compatible with the map
$\pi$, unless $\dim W_k\leq 1$ for all $k>1$.

Let us write $V$ for $W_1$.
Under the $\CC^\times$-action given by the weight $a$, the composition $V\hookrightarrow W\twoheadrightarrow V$ of
the inclusion with the projection onto $V$ is the identity and each map is $\CC^\times$-equivariant.
Taking linear duals gives the equivariant composition $V^*\hookrightarrow W^*\twoheadrightarrow V^*$, and this
induces the composition $\PP(V^*)\hookrightarrow\PP_a(W^*)\dasharrow\PP(V^*)$.
We obtain ordinary projective space $\PP(V^*)$ because $t\in\CC^\times$ acts as multiplication by $t$ on $V^*$.
We write $\defcolor{\pr_a}$ for the projection map $\PP_a(W^*)\dasharrow\PP(V^*)$, which is undefined on the image
of the annihilator of $V$ in  $\PP_a(W^*)$.  
In addition, we write $\defcolor{\pr}$ for the composition $\pr_a\circ\pi$.
We summarize these maps in the following commutative diagram:
\[
  \begin{picture}(205,73)(-34,0)
  \put(-5,63){$G_a\curvearrowright\PP^{n+m}$}
   \multiput(62,4)(6,0){6}{\line(1,0){4}}  \put(97,2){\includegraphics{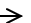}}

   \multiput(61,56)(6,-6){7}{\includegraphics{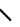}}
   \put(103,14){\includegraphics{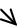}}

   \put(37,15){\line(0,1){44}}   \put(35,14.5){\includegraphics{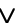}}

   \put(27,39){\scriptsize$\pi$}  \put(77,8){\scriptsize$\pr_a$} \put(90,39){\scriptsize$\pr$}
   \put(-23,1){$\PP_a^{n+m}\simeq\PP_a(W^*)$}
   \put(108,1){$\PP(V^*)\simeq\PP^n$.}

  \end{picture}
\]

Let $X\subset\PP(V^*)$ and $Z\subset\PP_a(W^*)$ be varieties such that $\pr_a$ is an isomorphism between $Z$ and $X$.  
In this case, a linear section $X\cap L$ pulls back along $\pr_a$ to $Z\cap\pr_a^{-1}(L)$.
We remark that the subvariety $\pr_a^{-1}(L)$, which is given by $d$ forms that are linear in
$x_0,\dotsc,x_n$, is not general.
For example, $\pr_a^{-1}(L)$ includes $\calV(x_0,\dotsc,x_n)$, which contains the singular locus of $\PP_a(W^*)$.  
We let $\defcolor{U}\subset\PP_a(W^*)$ be the open subset over which $\pi$ is a covering space.
For $u\in U$, $G_a$ acts freely on the fiber $\pi^{-1}(z)$. 
The following lemma relates $Z\cap\pr_a^{-1}(L)$ to $X\cap L$:

%%%%%%%%%%%%%%%%%%%%%%%%%%%%%%%%%%%%%%%%%%%%%%%%%%%%%%%%%%%%%%%%%%%%%%%%%%%%%%%%%
\begin{lemma}\label{L:ItAllWorksOut}
  Let $Z\subset\PP_a(W^*)$ be a subvariety of dimension $d$ such that $Z\cap U$ is dense in $Z$ and $\pr_a$ is an
  isomorphism between $Z$ and $\defcolor{X}\vcentcolon=\pr_a(Z)$.
  Let $\defcolor{Y}\vcentcolon= \pi^{-1}(Z)\subset\PP^{n+m}$ be its inverse image.
  Suppose that $L\subset\PP(V^*)$ is a general linear subspace of codimension $d$.
  Then,
  \begin{enumerate}[{\rm(\roman*)}]  
  \item $Z\cap\pr_a^{-1}(L)$ is transverse and $\pr_a\colon  Z\cap\pr_a^{-1}(L)\to X\cap L$ is a bijection.%\label{propertiesL1}
    
  \item $Y\cap\pr^{-1}(L)$ is transverse and $\pi\colon Y\cap\pr^{-1}(L)\to Z\cap\pr_a^{-1}(L)$ is a $|G_a|$ to $1$
    surjection.%\label{propertiesL2} 
  \item For any component $Y'$ of $Y$, $\pi\colon Y'\cap\pr^{-1}(L)\to Z\cap\pr_a^{-1}(L)$ is a
    $|\operatorname{Stab}_{G_a}(Y')|$ to $1$ surjection.\label{propertiesL3}
    
  \end{enumerate}
\end{lemma}
%%%%%%%%%%%%%%%%%%%%%%%%%%%%%%%%%%%%%%%%%%%%%%%%%%%%%%%%%%%%%%%%%%%%%%%%%%%%%%%%%

We note that $Y=\pi^{-1}(Z)$ may not be irreducible.
Each irreducible component, however, maps surjectively onto $Z$.

%%%%%%%%%%%%%%%%%%%%%%%%%%%%%%%%%%%%%%%%%%%%%%%%%%%%%%%%%%%%%%%%%%%%%%%%%%%%%%%%%
\begin{proof}
  We address transversality after establishing the set-theoretic assertions.
  For $x\in X\cap L$, let $z$ be the unique point of $Z$ with $\pr_a(z)=x$.
  Since $z\in\pr_a^{-1}(L)$, this completes the proof of the first statement.

  Let $z\in Z\cap\pr_a^{-1}(L)$.  By our assumptions, $Z\cap\pr_a^{-1}(L)\subset U$, so $z\in U$.
  Then $\pi^{-1}(z)\subset Y\cap\pi^{-1}\pr_a^{-1}(L)=Y\cap\pr^{-1}(L)$.
  The second statement follows as $\pi\colon Y\to Z$ is $|G_a|$ to 1 over points of $U$.
  
  For the third statement, we observe that $\pr^{-1}(L)$ is invariant under the $G_a$-action.
  Therefore, for all $g\in G_a$, $g.(Y'\cap \pr^{-1}(L))=(g.Y')\cap\pr^{-1}(L)$.
  The claim follows from the second statement and a counting argument.

  For transversality, let $x\in X\cap L$.
  As $L$ is general, this intersection is transverse and the forms defining $L$ generate the maximal ideal in the local
  ring of $X$ at $x$. 
  Transversality in the first statement 
  follows since the map $\pr_a$ is an isomorphism between $Z$ and $X$
  and $\pr_a^{-1}(L)$ is defined by the same forms as $L$.
  Transversality in the second statement also follows, since the maximal ideal of
  $Y$ at $y$ is generated by the pull  back of the maximal ideal of $Z$ at $\pi(y)$  and $\pr_a^{-1}(x)\in U$.
\end{proof}
%%%%%%%%%%%%%%%%%%%%%%%%%%%%%%%%%%%%%%%%%%%%%%%%%%%%%%%%%%%%%%%%%%%%%%%%%%%%%%%%%

While $\pr^{-1}(L)$ is a linear subspace, it is not general.
We need a result similar to Lemma~\ref{L:ItAllWorksOut} for a general linear subspace $\Lambda\subset\PP^{n+m}$.
We note that since $\Lambda$ is general, $\pi^{-1}(\pi(\Lambda))$ consists of a union of
$|G_a|$ linear subspaces.

%%%%%%%%%%%%%%%%%%%%%%%%%%%%%%%%%%%%%%%%%%%%%%%%%%%%%%%%%%%%%%%%%%%%%%%%%%%%%%%%%
\begin{lemma}
  Let $Z\subset\PP_a(W^*)$ be a subvariety of dimension $d$ such that $Z\cap U$ dense in $Z$.
  Let $Y\vcentcolon= \pi^{-1}(Z)\subset\PP^{n+m}$ be its inverse image, and suppose that
  $\Lambda\subset \PP^{n+m}$ is a 
  general linear subspace of codimension $d$.
  Then, $\pi\colon Y\cap\Lambda\to Z\cap\pi(\Lambda)$ is a bijection.
\end{lemma}
%%%%%%%%%%%%%%%%%%%%%%%%%%%%%%%%%%%%%%%%%%%%%%%%%%%%%%%%%%%%%%%%%%%%%%%%%%%%%%%%%
\begin{proof}
 Since $\Lambda$ is general, $Z\cap\pi(\Lambda)\subset U$.
 Suppose that $q,q'\in Y\cap\Lambda$ are in the same fiber of $\pi$, and let
 $g\in G_a$ be defined by $q'=g.q$.
 Since $Y$ is $G_a$-invariant, we have $q'\in Y\cap (g.\Lambda)$.
 Since $\Lambda$ is general, $Y\cap \Lambda\cap (g.\Lambda)$ is empty unless $g$ is the identity.
 Therefore, $q=q'$, and we conclude that $\pi$ is injective on $Y\cap\Lambda$.  

 This map is also surjective.
 If $p\in Z\cap\pi(\Lambda)$, then there is a point $q\in \pi^{-1}(p)\cap\Lambda$.
 As $Y=\pi^{-1}(Z)$, it contains $\pi^{-1}(p)$ and thus $q\in Y\cap\Lambda$ and $\pi(q)=p$.
\end{proof}
%%%%%%%%%%%%%%%%%%%%%%%%%%%%%%%%%%%%%%%%%%%%%%%%%%%%%%%%%%%%%%%%%%%%%%%%%%%%%%%%%

%%%%%%%%%%%%%%%%%%%%%%%%%%%%%%%%%%%%%%%%%%%%%%%%%%%%%%%%%%%%%%%%%%%%%%%%%%%%%%%%%
\subsection{Khovanskii bases and the degeneration}

Let $X$ be a $d$-dimensional complex variety and $V\subset\CC(X)$ a finite-dimensional complex vector subspace.
Suppose that the image of $V^\times$ under $\nu$ generates $\ZZ^d$ and $V$ has a finite Khovanskii basis $\calB$ such that
$\calB\not\subset Vs$.
For each $k\in\NN$, let $\defcolor{W_k}s^k\vcentcolon=\operatorname{Span}(\calB\cap V^ks^k)\subset V^ks^k$ be the span of
the elements of $\calB$ of homogeneous degree $k$.
We define $W\vcentcolon=\bigoplus_{k\geq 1}W_k$ where $V=W_1$ and construct the corresponding weighted projective space as in
Section~\ref{sec:weightedReview}.
Anderson's toric degeneration~\cite{Anderson} naturally embeds into $\PP_a(W^*)$.
The weighted projective space $\PP_a(W^*)$ is needed (rather than $\PP(V^*)$) to accommodate the generators of
$\gr R(V)\simeq\CC[S(V,\nu)]$ which are not in $V$, as these are needed for embedding the toric fiber.

We introduce coordinates by ordering the elements of $\calB=\{b_0s^{a_0},\dotsc,b_{n+m}s^{a_{n+m}}\}$ where
$a_0=\dotsb=a_n=1$, and for $n<j\leq n{+}m$, $a_j>1$.
Necessarily, $\{b_0,\dotsc,b_n\}\subset V$, since $Vs$ generates $R(V)$.
Then, for each $n<j\leq n{+}m$, there is a homogeneous polynomial $\defcolor{h_j}\in\CC[z_0,\dotsc,z_n]$ of
degree $a_j$ such that $b_j=h_j(b_0,\dotsc,b_n)$.

Using the Khovanskii basis $\calB$, the
Kodaira map to $\PP_a(W^*)$ from $X=\Proj(R(V))$
is $\varphi_\calB\colon z\mapsto [b_0(z),\dotsc,b_{n+m}(z)]$.
Since, for $n<j\leq n+m$, $b_j=h_j(b_0,\dotsc,b_n)$, the  image of $\varphi_\calB$ is a graph over the the image of
$\varphi_V$ in $\PP(V^*)\subset\PP_a(W^*)$.

The constructions of $I_\calB$, $\calA$, $w$,
$\ini_{\calA}(I_\calB)$, and $\calG_t$ from Section~\ref{sec:embedding} all carry over to this general case since all of
these ideals are $a$-homogeneous.
Collectively, they embed Anderson's toric
degeneration into the weighted projective space $\PP_a(W^*)$.
The special fiber $\calX_0$ is a toric variety with ideal $\calG_0$ and toric Kodaira map $\varphi_{p,\calA}$, where
$p\in\calX_0\cap\TT_a$ (as before, the torus $\TT_a\subset\PP_a(W^*)$ consists of those points with nonzero  coordinates). 

We pull back the embedded toric degeneration $\calX\subset\PP_a(W^*)\times\CC_t$ along $\pi$
to obtain a flat family $\defcolor{\calY}\subset\PP^{n+m}\times\CC_t$ that is a toric degeneration in the sense of
Remarks~\ref{Ref:ToricComplex} and~\ref{Ref:toriccomplex:example} as $\calY$ or $\calY_0$ may not be irreducible.
We explain how the equations defining the family $\calY$ may be obtained.
Let $\CC[y_0,\dotsc,y_{n+m}]$ be the homogeneous coordinate ring of the projective space $\PP^{n+m}$.
The map $\pi\colon\PP^{n+m}\to\PP_a(W^*)$ corresponds to the map
$\pi^*\colon\CC[x_0,\dotsc,x_{n+m}]\to\CC[y_0,\dotsc,y_{n+m}]$
induced by $x_i\mapsto y_i^{a_i}$.
Let
 \begin{equation}\label{Eq:calF_t}
   \defcolor{\calF_t}\vcentcolon= \{ \pi^*(g_t) :  g_t\in\calG_t\}
 \end{equation}
be the pull back of the equations $\calG_t$ for the embedded degeneration $\calX\to\CC_t$.
Then $\calY=\calV(\calF_t)\subset\PP^m\times\CC_t$. 
This lifted family $\calY\to\CC_t$ is the fiberwise pull back of Anderson's toric degeneration $\calX\to\CC_t$ along the finite map $\pi$, where $G_a$ acts on $\calY$ fiberwise.

%%%%%%%%%%%%%%%%%%%%%%%%%%%%%%%%%%%%%%%%%%%%%%%%%%%%%%%%%%%%%%%%%%%%%%%%%%%%%%%%%
\subsection{Weighted Khovanskii homotopy}

We explain how to use the embedded degeneration $\calX$ in $\PP_a(W^*)$ to compute the linear section $\varphi_V(X)\cap L$.
Since $\calX_1=\varphi_\calB(X)$, it is natural to propose to compute $\calX_1\cap\pr_a^{-1}(L)$ using an adaptation of the 
linear section homotopy to weighted projective space by following points of $\calX_0\cap\pr_a^{-1}(L)$ along Anderson's
degeneration.
Unfortunately, $\pr_a^{-1}(L)$ is not sufficiently general for the toric special fiber in Anderson's
degeneration. 

To avoid this problem, we pull back the toric degeneration $\calX$ along $\pi$ to $\calY$ and
use a linear section homotopy to compute the linear section $\calY_1\cap\pr^{-1}(L)$.
Since $\pr^{-1}(L)$ is not a general linear subspace, we
instead choose a general linear subspace $\Lambda\subset\PP^{n+m}$ of codimension $d$.
Next, we use Algorithm~\ref{Alg:TDA} to compute $\calY_1\cap\Lambda$, which is a witness set for $\calY_1$.
Then, we use the witness set homotopy (Algorithm~\ref{Alg:WSH}) to compute  $\calY_1\cap\pr^{-1}(L)$.
Finally, $\varphi_V(X)\cap L$ is computed as $\pr(\calY_1\cap\pr^{-1}(L))$.

%%%%%%%%%%%%%%%%%%%%%%%%%%%%%%%%%%%%%%%%%%%%%%%%%%%%%%%%%%%%%%%%%%%%%%%%%%%%%%%%%
\begin{algorithm}[Weighted Khovanskii homotopy algorithm]\
\label{Alg:WeightedKhovanskii}

  {\bf Input:}   A finite-dimensional subspace $V\subset\CC(X)$ for a variety $X=\Proj(R(V))$ of

  \mbox{{\color{white}{\bf Input:}}}
  dimension $d$, finite Khovanskii basis $\calB\not\subset Vs$ for $V$, and a general linear

  \mbox{{\color{white}{\bf Input:}}} subspace $L\subset\PP(V^*)$ of codimension $d$. 

  {\bf Output:}
   Points in the linear section $\varphi_V(X)\cap L$ in the projective space $\PP(V^*)$.

\pagebreak
  {\bf Do:}
\begin{enumerate}[(\roman*)]
\item  Follow Steps \ref{Alg:Khovanskii:Step2} through~\ref{Alg:Khovanskii:Step6} of
  Algorithm~\ref{Alg:Khovanskii}, {\it mutatis mutandis}:
  The ideal $I_\calB$ is the kernel of the map $\CC[x_0,\dots,x_{n+m}]\rightarrow R(V)$ where $x_i\mapsto b_i s^{a_i}$.
  
\item Pull back the family $\calX$ along $\pi$ to compute the family $\calY$ defined by $\calF_t$,
  see Definition~\eqref{Eq:calF_t}.
  
\item Compute Kodaira maps for each irreducible component of $\calY_0$.\label{Alg:WeightedKhovanskii:Step3}
       
\item Let $\Lambda\subset\PP^{n+m}$ be a general linear subspace of codimension $d$ 
  and use Algorithm~\ref{Alg:TDA} to compute $\calY_1\cap\Lambda$.%\label{Alg:WeightedKhovanskii:Step4}
  
\item Use Algorithm~\ref{Alg:WSH} to compute $\calY_1\cap\pr^{-1}(L)$.
\item Return $\varphi_V(X)\cap L=\pr(\calY_1\cap\pr^{-1}(L))$.
\end{enumerate}
\end{algorithm}
%%%%%%%%%%%%%%%%%%%%%%%%%%%%%%%%%%%%%%%%%%%%%%%%%%%%%%%%%%%%%%%%%%%%%%%%%%%%%%%%%

%%%%%%%%%%%%%%%%%%%%%%%%%%%%%%%%%%%%%%%%%%%%%%%%%%%%%%%%%%%%%%%%%%%%%%%%%%%%%%%%%
\begin{remark}\label{Rem:WPSAlgorithm}
  We discuss Step \ref{Alg:WeightedKhovanskii:Step3}
 of  Algorithm \ref{Alg:WeightedKhovanskii}.
 As $\calY_0$ may consist of several components and $\calY_0=\pi^{-1}(\calX_0)$, the group $G_a$ acts transitively on
 these components.
 Moreover, each component is a projective toric variety $X_{q,\calC}$ for a point $q\in\calY_0\cap\TT$ and
 all have the same set of exponents, which are the columns of matrix $\calC$.
 We explain how to compute both $q$ and $\calC$.

From Step~\ref{Alg:Khovanskii:Step6} of Algorithm~\ref{Alg:Khovanskii}, we have a toric Kodaira map
$\varphi_{p,\calA}\colon(\CC^\times)^d\to\PP^{n+m}_a$ such that $\calX_0=X_{p,\calA}$.
The image includes the point $p\in\calX_0\cap\TT_a$.
Points $q\in\pi^{-1}(p)$ are obtained by taking all $a_j$-th roots of the coordinate $p_j$ of
$p$, for all $j$,
\[
  \pi^{-1}(p)=\{q\in\PP^{n+m}: q_j^{a_j}=p_j \,\mbox{ for }\, j=0,\dotsc,n{+}m{+}1 \}.
\]
  
  It remains to determine the exponents $\calC$ for $\pi^{-1}(X_\calA)$.
  As in Remark~\ref{R:calBKodaira}, we have a basis $u_1,\dotsc,u_{n+m-d}\in\ZZ^{n+m+1}$ for $\ker(\calA)$. 
  These vectors give equations $x^{u_i}=1$ for $X_\calA\cap\TT_a$.
  Applying $\pi^*$ substitutes $y_j^{a_j}$ for $x_j$ and gives equations for $\pi^{-1}(X_\calA)\cap\TT$,
  \begin{equation}\label{Eq:calY_0}
    y^{v_i}=1\qquad i=1,\dotsc,n{+}m{-}d\,,
  \end{equation}
  where $v_i$ is obtained from $u_i$ by multiplying its $j$th coordinate by $a_j$.
  
  The System~\eqref{Eq:calY_0} for $\pi^{-1}(X_\calA)\cap\TT$ leads to equations for
  $\defcolor{Y}\vcentcolon=\pi^{-1}(X_\calA)$, which 
  form a lattice ideal~\cite[\sect\ 2]{EisenbudSturmfels} for the lattice $K$ spanned by $\{v_1,\dotsc,v_{n+m-d}\}$.
   That is, $Y=\calV\langle y^\alpha-y^\beta\mid \alpha-\beta\in K\rangle$.
    
  Let \defcolor{$(\gamma_1,\dotsc,\gamma_d,\bbI)$} be a basis for the annihilator of $K$ in $\ZZ^{n+m+1}$.
  Suppose that $\calC$ is the $d\times(n{+}m{+}1)$ matrix whose rows are $\gamma_1,\dotsc,\gamma_d$.
 Then $X_\calC$ is the component of  $Y$ containing the identity $\bbI\in\TT$.
 We remark that $\calC$ may be computed from $v_1,\dotsc,v_{n+m-d}$ using the Hermite normal form.
 All Kodaira maps needed in Step \ref{Alg:WeightedKhovanskii:Step3}
 of  Algorithm \ref{Alg:WeightedKhovanskii} can then be computed by translations. \hfill$\diamond$ 
\end{remark}
%%%%%%%%%%%%%%%%%%%%%%%%%%%%%%%%%%%%%%%%%%%%%%%%%%%%%%%%%%%%%%%%%%%%%%%%%%%%%%%%%

%%%%%%%%%%%%%%%%%%%%%%%%%%%%%%%%%%%%%%%%%%%%%%%%%%%%%%%%%%%%%%%%%%%%%%%%%%%%%%%%%
\begin{remark}
  The number of components of $\calY$ or of $\calY_0$ impacts the number of
   Kodaira maps needed in Step~\ref{Alg:WeightedKhovanskii:Step3} of  Algorithm \ref{Alg:WeightedKhovanskii}.
   Reductions in the number of Kodaira maps may significantly improve the efficiency of the
   algorithm. 

 When $\calY$ is known to be reducible, this structure may be exploited, as
  Statement \ref{propertiesL3} of Lemma \ref{L:ItAllWorksOut} implies that it is enough to apply Algorithm
  \ref{Alg:WeightedKhovanskii} to a single component of $\calY$.
  In particular,
   the map $\pi$ sends the curves in a linear section of one component onto
  $\calX\cap\pr^{-1}(L)$.

 When $\calY_0$ has fewer than $|G_a|$ components, 
then there are redundant Kodaira maps constructed in Remark \ref{Rem:WPSAlgorithm}.
More precisely, the number of redundant maps is the number of points of $\pi^{-1}(p)$ in a component of $\calY_0$.
We provide details on computing non-redundant Kodaira maps, assuming, as in Remark \ref{Rem:WPSAlgorithm}, that
  $\bbI\in\calY_0$.  The general case is obtained by translation.
    Let
  \[
\defcolor{\sat(K)}:= \{ w\in\ZZ^{n+m+1}: rw\in K\text{ for some }0\neq r\in\ZZ\}
  \]
be the saturation of $K$ and $M=\ker(\bbI)\subset\ZZ^{n+m+1}$.
  We note that $\sat(K)\subset M$.
  We identify $\TT$ with $\Hom(M,\CC^\times)$ so that $\calY_0\cap\TT=\Hom(M/K,\CC^\times)$, as these are the points
  satisfying System~\eqref{Eq:calY_0}.
  The component of $\calY_0\cap\TT$ containing the identity $\bbI\in\TT$ is $\Hom(M/\sat(K),\CC^\times)$, and the group of
  components of $\calY_0\cap\TT$ is $\Hom(\sat(K)/K,\CC^\times)$.
  Hence, the elements of $\Hom(\sat(K)/K,\CC^\times)$ generate Kodaira maps to distinct components of $\calY_0$.
   \hfill{$\diamond$}   
\end{remark}
%%%%%%%%%%%%%%%%%%%%%%%%%%%%%%%%%%%%%%%%%%%%%%%%%%%%%%%%%%%%%%%%%%%%%%%%%%%%%%%%%

%%%%%%%%%%%%%%%%%%%%%%%%%%%%%%%%%%%%%%%%%%%%%%%%%%%%%%%%%%%%%%%%%%%%%%%%%%%%%%%%%
\begin{proof}[Proof of correctness of Algorithm {\ref{Alg:WeightedKhovanskii}}]
  We need only show that the tracked paths provide enough points to compute
    $\varphi_V(X)\cap L$.
By Statement \ref{propertiesL3} of Lemma \ref{L:ItAllWorksOut}, for each $t$,
the map $\pi:\calY_t\cap\Lambda\rightarrow\calX_t\cap\pi(\Lambda)$ is a bijection.

The polyhedral homotopy correctly computes the points of $\calY_0\cap\Lambda$.
By Theorem~\ref{T:TDHA}, Algorithm~\ref{Alg:TDA} correctly computes the points of $\calY_1\cap\Lambda$.
Since the solution paths of the homotopy $\calX\cap\pi(\Lambda)$ are disjoint, the solution paths of
$\calY\cap\Lambda$ lie above paths of $\calX\cap\pi(\Lambda)$. 
In fact, by Statement \ref{propertiesL3} of Lemma \ref{L:ItAllWorksOut}, $\pi$ is a bijection between these sets of paths.
Therefore, there is a bijection between the ends of the homotopy paths of $\calY\cap\Lambda$ and points in
$\calX_1\cap\pi(\Lambda)$.
The correctness of the final computation then follows from the correctness of Algorithm \ref{Alg:WSH}.
\end{proof}
%%%%%%%%%%%%%%%%%%%%%%%%%%%%%%%%%%%%%%%%%%%%%%%%%%%%%%%%%%%%%%%%%%%%%%%%%%%%%%%%%

%%%%%%%%%%%%%%%%%%%%%%%%%%%%%%%%%%%%%%%%%%%%%%%%%%%%%%%%%%%%%%%%%%%%%%%%%%%%%%%%%
\begin{example}%\label{Ex:Cubics}
 Let $V$ be the space of cubic polynomials in $\CC[x,y]$ which vanish at
 the  points $(4,4), (-3,-1), (-1,-1)$ and $(3,3)$.
 (This example is related to the example of~\cite[\sect~5.1]{DHS}, which considers quartics vanishing at these points.)
 Then $V$ is six-dimensional with a basis:
 \begin{multline*}
   \{b_0, \dots, b_5\} = \{\underline{xy}-y^2+x-y, \,\underline{x^2}-y^2+4x-4y,  \,\underline{y^3}-6y^2+5y+12, \\
   \underline{x y^2}-6y^{2}-x+6y+12,\, \underline{x^2y}-6y^{2}-4x+9y+12,\,\underline{x^3}-6y^2-13x+18y+12\}.
 \end{multline*}
 We compute a general linear section of $X = \Proj(R(V))$ in $\PP(V^*)=\PP^5$ with Algorithm~\ref{Alg:WeightedKhovanskii}.
 Let $\succeq$ be the order on $\ZZ^2$ where $(a,b)\succeq(c,d)$ if $a+b<c+d$ or else $a+b=c+d$ and $a<c$.
 Define a valuation $\nu$ on $\CC(X)=\CC(x,y)$ as follows: for $f\in\CC[x,y]$, $\nu(f)=(a,b)$ where $(a,b)$ is the
   $\succeq$-minimal exponent of a term of $f$.
 This order and valuation $\nu$ are compatible with the \texttt{grevlex} order $\leq$ on $\CC[x,y]$ with $x>y$ in
 that $(a,b)\succeq(c,d)$ if and only if $x^ay^b\leq x^cy^d$.
 Using the subduction algorithm, as implemented in the unreleased \texttt{Macaulay2} package
 \texttt{SubalgebraBases}~\cite{SubalgebraBasesPackage} 
 applied to $\{b_0s, \dots, b_5s\}$, we obtain a Khovanskii basis $\calB = \{b_0s, \dots, b_5s,\, b_6s^2, b_7s^3\}$ with 
 two additional generators, where 
 \begin{eqnarray*}
   b_6 &\vcentcolon=& \underline{xy^3}-y^4+10x^2y-26xy^2+16y^3+10x^2-15xy+5y^2+12x-12y, \ \mbox{ and}\\
   b_7 &\vcentcolon=& \underline{10x^4y}-49x^{3}y^{2}+89x^{2}y^{3}-71xy^{4}+21y^{5}+10x^4-18x^3y-18x^2y^2 \\
       &&+50xy^3-24y^4+31x^3-83x^2y+73xy^2-21y^3+24x^2-48xy+24y^2.
 \end{eqnarray*}
The corresponding matrix of valuations is
 \[
    \calA = \nu(\calB) = \left(\begin{matrix}
     1&2&0&1&2&3&1&4\\
     1&0&3&2&1&0&3&1\\
     1&1&1&1&1&1&2&3 \end{matrix}\right).
\]
The Newton-Okounkov body, as displayed in Figure~\ref{F:NOB}, is obtained by intersecting the cone generated by the
columns of $\calA$ with the hyperplane where the third coordinate is 1.
The vertices $(1/2,3/2)$ and $(4/3,1/3)$ come from the initial (underlined) terms of $b_6$ and $b_7$. 
While they are not integers, the Newton-Okounkov body has normalized volume 5, which is the degree of $X$.
We may interpret this volume as follows:
Two cubics drawn from $V$ meet in $5=3^2-4$ points outside the base locus $\calV(V)=\{(4,4), (-3,-1), (-1,-1), (3,3)\}$.

%%%%%%%%%%%%%%%%%%%%%%%%%%%%%%%%%%%%%%%%%%%%%%%%%%%%%%%%%%%%%%%%%%%%%%%%%%%%%%%%%
\begin{figure}[htb]
  \centering
  \begin{picture}(164,80)(-59,0)
      \put(0,0){\includegraphics{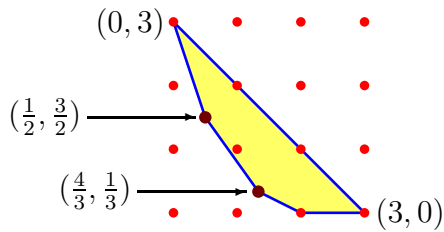}}
      \put(-26,72){$(0,3)$}
      \put(-29,39.5){\vector(1,0){41}}\put(-59,37){$(\frac{1}{2},\frac{3}{2})$}
      \put(-10,11.5){\vector(1,0){41}}\put(-40, 9){$(\frac{4}{3},\frac{1}{3})$}
      \put(80,0){$(3,0)$}
  \end{picture}
  \caption{Newton Okounkov body for the space of cubic polynomials vanishing at
     $(4,4), (-3,-1), (-1,-1)$ and $(3,3)$.}\label{F:NOB}
\end{figure}
%%%%%%%%%%%%%%%%%%%%%%%%%%%%%%%%%%%%%%%%%%%%%%%%%%%%%%%%%%%%%%%%%%%%%%%%%%%%%%%%%

The weight $w =(-6, -5,  0)$ is compatible with the \texttt{grevlex} order $\leq$ on $\calB$ in that for $b\in\calB$, 
the $\leq$-leading term has lowest $w$-weight, so that $\lt_\leq b = b_w$.
Choosing a term order on $\CC[x_0,\dotsc,x_7]$ that is compatible with $w\calA$, we use \texttt{Macaulay2} to compute a
Gr\"obner basis $\calG$ for $I_\calB$.
This basis consists of $17$ polynomials which are $\defcolor{a}\vcentcolon=(1,1,1,1,1,1, 2, 3)$-homogeneous.
Let $\CC^\times$ act on $a$-homogeneous polynomials using $w\calA$ in place of $w$ in Formula~\eqref{Eq:w-action}.
Then we compute $\calG_t\vcentcolon=\{g_t : g\in\calG\}$ as in Formula~\eqref{Eq:f_t}, which defines a flat family
$\calX\subset\PP^7_a\times\CC_t$ with toric special fiber $\calX_0$.
This family pulls back along $\pi\colon\PP^7\to\PP^7_a$ to a family $\calY\subset\PP^7\times\CC_t$.
The pull back  $\calY_0$ of $\calX_0$ under $\pi$ is a toric variety as it is irreducible.
From Remarks~\ref{R:calBKodaira} and \ref{Rem:WPSAlgorithm}, a Kodaira map for $\calY_0$ is 
 \begin{align*}\varphi_{p, \calA} \colon (\CC^*)^2 &\longrightarrow \PP^7 \\
   z &\longmapsto [z_1^6 z_2^3, \,z_1^4 z_2^3, \,z_1^6, \,z_1^4, \,z_1^2, \,1, \,z_1^7 z_2^3, \,
       \sqrt[3]{10} \, z_1^6 z_2^4].
\end{align*}

The polyhedral homotopy finds $30$ points in $\calY_0\cap\Lambda$.
An application of the toric two-step algorithm (Algorithm \ref{Alg:TDA}) tracks these points to
$\calY_1\cap {\Lambda}$ with no paths diverging.
Then, the witness set homotopy algorithm (Algorithm \ref{Alg:WSH}) moves $\Lambda$ to $\pr^{-1}(L)$ and finds the points of 
$\calY_1 \cap \pr^{-1}(L)$. 
These $30$ points project under $\pi\colon\PP^7\to\PP^7_a$ to five points in $\calX_1\cap\ pr^{-1}(L)$.
Finally, applying the map $\pr\colon\PP^7\to\PP^5=\PP(V^*)$ gives all five
points in $\varphi_V(X)\cap L$. \hfill$\diamond$
\end{example}

%
%%%%%%%%%%%%%%%%%%%%%%%%%%%%%%%%%%%%%%%%%%%%%%%%%%%%%%%%%%%%%%%%%%%%%%%%%%%%%%%%% 

\section{Practical Considerations }\label{S:Exp}

We discuss how to compute a finite Khovanskii basis as well as options for tracking overdetermined homotopy systems.

%%%%%%%%%%%%%%%%%%%%%%%%%%%%%%%%%%%%%%%%%%%%%%%%%%%%%%%%%%%%%%%%%%%%%%%%%%%%%%%%%
\subsection{Computing a Khovanskii basis.}
Whether or not a given vector space $V$ of functions has a finite Khovanskii basis is generally not known and
depends on the choice of valuation.
Given $V$ and a valuation $\nu$, the subduction
algorithm~\cite[\algo~2.18]{KM_tropical} terminates and returns a finite Khovanskii basis, when one exists.
We only know of implementations when $V$ is a space of polynomials and $\nu$ is induced by a term
order~\cite{M2,SubalgebraBasesPackage}.
These also compute a SAGBI basis \cite{KM89,RS90}.

%%%%%%%%%%%%%%%%%%%%%%%%%%%%%%%%%%%%%%%%%%%%%%%%%%%%%%%%%%%%%%%%%%%%%%%%%%%%%%%%%
\subsection{Homotopy continuation for overdetermined systems.} 
Algorithms~\ref{Alg:Khovanskii} and~\ref{Alg:WeightedKhovanskii} generate a homotopy $(\calG_t,L)$
from a Gr\"obner basis $\calG$ defining $\calX$.
This is not typically square in that it has more equations than variables.
As most implementations of homotopy continuation, including the user-defined
homotopy in \texttt{Bertini}, 
require square systems, we need to choose a square subsystem for tracking from $t=0$ to $t=1$.

Typically, a square subsystem is obtained by taking linear combinations of elements in a given system.
  There is an alternative for equations $\calG_t$ from a toric degeneration.
  Let $\calA$ be the matrix of exponents defining the Kodaira map for the toric special fiber, $\calX_0$.
  The
  intersection $\calX_0\cap\TT$ with the dense torus of $\PP^n$ is the complete intersection defined by
  binomials $x^{u_i}-c_ix^{v_i}$ for $i=1,\dotsc,n{-}d$ such that $\{u_i-v_i : i=1,\dotsc,n{-}d\}$ form a basis for
  $\ker(\calA)$.
  The points of $\calX_0\cap L$ are smooth isolated solutions to the square system given by these binomials and
  the linear forms defining $L$.
 If we choose $\calF_t\subset\calG_t$ to consist of $n{-}d$ elements whose leading
  binomials are $x^{u_i}-c_ix^{v_i}$, then $(\calF_t,L)$ is a square subsystem of $(\calG_t,L)$ which
  defines curves containing  $\calX_0\cap L$ and,
  therefore, is sufficient for homotopy continuation.

%%%%%%%%%%%%%%%%%%%%%%%%%%%%%%%%%%%%%%%%%%%%%%%%%%%%%%%%%%%%%%%%%%%%%%%%%%%%%%%%%
\bibliographystyle{amsplain}
\bibliography{bibl}

\end{document}